\documentclass[preprint,12pt]{elsarticle}

\usepackage{amsthm}
\usepackage{amssymb}
\usepackage{fnpct}
\usepackage{amsmath}
\usepackage{epsfig}
\usepackage{framed}
\usepackage{etoolbox}
\usepackage{fnpct}
\usepackage{overpic}
\usepackage{url}
\usepackage[font=small,labelfont=bf]{caption} 
\usepackage{relsize} 
\usepackage{setspace}
\usepackage{tikz}
\usepackage{pdfpages}

\usepackage{enumitem}

\usetikzlibrary{arrows,positioning}

\usepackage[linesnumbered,ruled]{algorithm2e}

\usepackage[utf8]{inputenc}

\newcommand{\step}{0.6cm}

\newcommand\numberthis{\addtocounter{equation}{1}\tag{\theequation}}

\DeclareMathOperator*{\Spg}{spg}

\DeclareMathOperator*{\Var}{Var}

\def\hide #1 {}
\long\def\longhide #1 {}

\usepackage{amsthm}

\theoremstyle{plain}
\newtheorem{theorem}{Theorem}[section]
\newtheorem*{theorem*}{Theorem}
\newtheorem{lemma}[theorem]{Lemma}

\newtheorem{conjecture}[theorem]{Conjecture}

\theoremstyle{definition}
\newtheorem{remark}[theorem]{Remark}

\newtheorem{definition}[theorem]{Definition}

\theoremstyle{definition}

\newcommand*{\Scale}[2][4]{\scalebox{#1}{$#2$}}%

\newcounter{nootje}
\renewcommand{\check}[1]
  {\marginpar{\tiny\begin{minipage}{20mm}\begin{flushleft}\thenootje : #1\end{flushleft}\end{minipage}}\addtocounter{nootje}{1}}
\journal{SIAM Journal on Imaging Sciences}

\begin{document}

\begin{frontmatter}



\title{Robust Phase Retrieval Algorithm for Time-Frequency Structured Measurements\footnote{Results of this paper were included in the PhD thesis of Palina Salanevich. Algorithm~\ref{noiseless_case_alg} and recovery guarantees in the case of noiseless measurements (Theorem~\ref{noiseless_reconstruction_algorithm}) were introduced in \cite{salanevich}. Section \ref{algorithm} contains proofs of these results. Also, a uniform bound for the number of large frame coefficients (Theorem \ref{number_of_big_meas_uniform}) was introduced in \cite{salanevich2017geometric} and is included here for completeness of the discussion. The main part of this paper, namely, the results on the order statistics of frame coefficients for Gabor frames and on robustness of the algorithm to noise in the measurements, are new and did not appear anywhere else.}}


\author[pfander]{G\"{o}tz E. Pfander}
\author[salanevich]{Palina Salanevich}

\address[pfander]{Department of Scientific Computing, Catholic University of Eichstätt-Ingolstadt, Germany. Email: pfander@ku.de}

\address[salanevich]{Department of Mathematics, University of California, Los Angeles, US. \\Email: psalanevich@math.ucla.edu.}

\begin{abstract}
We address the problem of signal  reconstruction from intensity measurements with respect to a measurement frame. This non-convex inverse problem is known as \emph{phase retrieval}. The case considered in this paper concerns phaseless measurements taken with respect to a Gabor frame. It arises naturally in many practical applications, such as diffraction imaging and speech recognition. We present a reconstruction algorithm that uses a nearly optimal number of phaseless time-frequency structured measurements and discuss its robustness in the case when the measurements are corrupted by noise. We show how geometric properties of the measurement frame are related to the robustness of the phaseless reconstruction. \hide{The presented algorithm is based on the idea of polarization as proposed by Alexeev, Bandeira, Fickus, and Mixon \cite{mixon1}.}
\end{abstract}


\end{frontmatter}


\section{Introduction}\label{Intro}

The \emph{phase retrieval} problem arises naturally in many applications within a variety of fields in science and engineering. Among these applications are optics \cite{Mill}, astronomical imaging \cite{Dain}, and microscopy \cite{Miao}.\par

As an example, let us consider the diffraction imaging problem \cite{Bunk}. To investigate the structure of a small particle, such as a DNA molecule, we illuminate the particle with X-rays and then measure the radiation scattered from it. When X-ray waves pass by an object and are measured in the far field, detectors are not able to capture the phase of the waves reaching them, but only their magnitudes. The measurements obtained in this way are of the form of pointwise squared absolute values of the Fourier transform of the object $x$, that is, the measurement map $\mathcal{A}$ is given by $\mathcal{A}(x) = \{|\mathcal{F}(x)(n)|^2\}_{n\in \Omega}$ where $\Omega$ is the sampling grid. Since $\mathcal{A}$ is not injective, some additional a priori information on the object $x$ is needed for reconstruction. For instance, knowledge of the chemical interactions between parts of a DNA molecule were used for the construction of the DNA double helix model in the Nobel Prize winning work of Watson, Crick, and Wilkins~\cite{Watson}.\par

One way to overcome this non-injectivity when no a priori information is available is \emph{masking}. To modify the phase front, one can insert a known mask after the object, as shown on Figure \ref{fig_diff_imaging}. The measurement map in this case is given by $\mathcal{A}(x) = \{|\mathcal{F}(x\odot f_t)(n)|^2\}_{n\in \Omega, t\in I}$, where $f_t$, $t\in I$, are the masks used, and $\odot$ denotes pointwise multiplication. By increasing the number of measurements in this way, we reduce the ambiguity in the reconstruction of signal $x$.\par 
\medskip
\begin{figure}[t]
\centering
\begin{overpic}[width=0.85\linewidth,tics=10,
]
{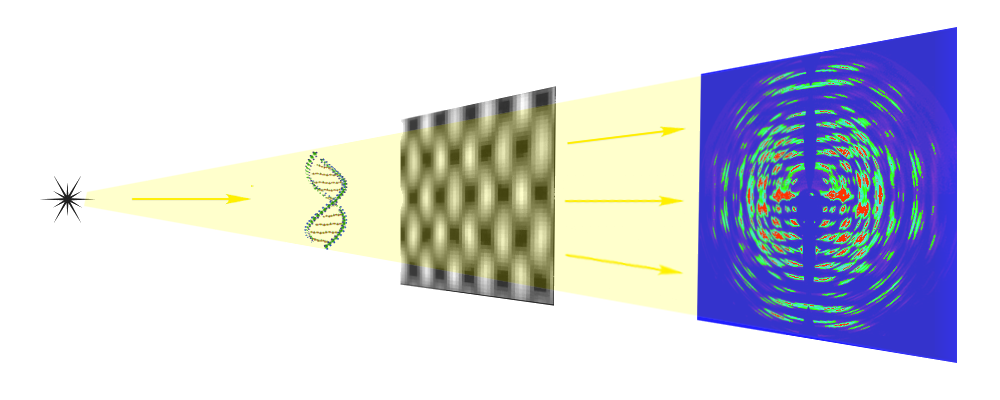}
\put(45,6){\footnotesize mask}
\put(30,11){\footnotesize DNA}
\put(28,9){\footnotesize molecule}
\put(0,14){\footnotesize X-ray source}
\put(73,1){\footnotesize diffraction pattern}
\end{overpic}
\caption{\label{fig_diff_imaging}A typical setup for structured illuminations in diffraction imaging using a phase mask.}
\end{figure}
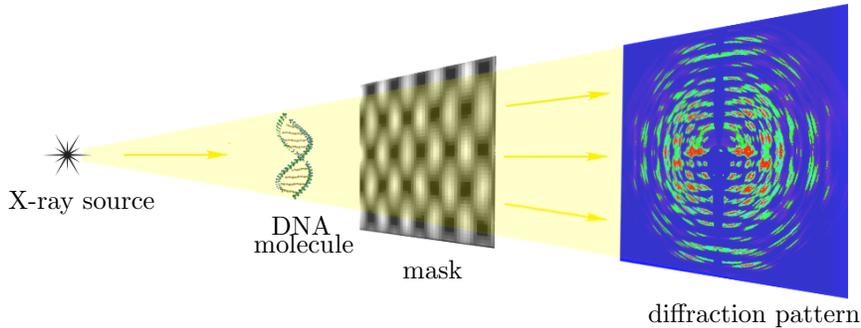

\longhide{
In this paper, we consider one-dimensional finite length signals. A justification for this is the fact that we are going to talk about reconstruction algorithms which ultimately operate with finite dimensional digital data. To translate a continuous signal into a finite dimensional vector, one can use sampling.\par }

Since the problem of signal reconstruction from magnitudes of Fourier coefficients is particularly hard to handle, a more general frame theoretical setting is frequently considered. Namely, for $\Phi = \{\varphi_j\}_{j = 1}^N \subset \mathbb{C}^M$ being a \emph{frame}, that is, a possibly overcomplete spanning set for $\mathbb{C}^M$, we aim to recover $x\in \mathbb{C}^M$ from its phaseless squared frame coefficients $\mathcal{A}_{\Phi}(x) = \{|\langle x, \varphi_j\rangle|^2\}_{j = 1}^N$.

Note that the masked Fourier coefficients of the signal $x$ with masks $\{f_t\}_{t\in I}\subset \mathbb{C}^M$ can be viewed as the frame coefficients of $x$ with respect to the frame $\Phi$ given by $\Phi = \{\varphi_{t,j}\}$, where $\varphi_{t,j}(m) = \frac{e^{2\pi i j m/M}}{\sqrt{M}}\bar{f_t}(m)$.\par

It is clear that, even in an optimal setting, $x$ can be reconstructed from intensity measurements only up to a global phase. Indeed, for every real $\theta$, the signals $x$ and $e^{i\theta} x$ produce the same intensity measurements. Thus, the goal of phase retrieval is to reconstruct the equivalence class $[x]\in\mathbb{C}/_\sim$ of $x$, where $x\sim y$ if and only if $x=e^{i\theta} y$ for some $\theta\in [0,2\pi)$. In the sequel, we are going to identify $x$ with its equivalence class $[x]\in\mathbb{C}/_\sim$.\par

Obviously, not every frame gives rise to an injective measurement map. But even in the case when $\mathcal{A}_{\Phi}$ is known to be injective, the problem of reconstructing $[x]$ from $\mathcal{A}_{\Phi}(x)$ is NP-hard in general \cite{Sahin}. So, the main goals in this area of applied mathematics is to find conditions on the number of measurements $N$ and vectors $\varphi_j$ for which there exist an efficient and robust numerical recovery algorithm.\par

Until recently, very little was known on how to achieve robust and efficient reconstruction given injectivity. Many practical methods used today have their origins in the alternating projection algorithms proposed in the 1970s by Gerchberg and Saxton \cite{gerchberg1972practical}. Due to their lack of global convergence guarantees, the problem of developing fast phase retrieval algorithms which have provable recovery and robustness guarantees receives significant attention today. Some of the most prominent suggested algorithms are PhaseLift \cite{balan2, candes1, candes2, candes3, candes4}, Wirtinger flow algorithms \cite{candes2015phase, Chen}, fast phase retrieval algorithm from local correlation measurements \cite{iwen2015fast}, and phase retrieval with polarization \cite{mixon1, mixon2}. The latter is described in more detail in Section \ref{polarization_idea}.\par

While recovery guarantees have been established for all these algorithms, most of them require the measurement frame vectors to be independent random vectors. Since measurements of this type are not implementable in practice, the design of fast and stable recovery algorithms with a small number of application relevant, structured, measurements remains an important problem. We address this problem below. \par

\longhide{\subsection{Relation to previous work}

Until recently, very little was known on how to achieve stable and efficient reconstruction given injectivity.  Many methods in use today have their origins in the alternating projection algorithms proposed in the 1970s by Gerchberg and Saxton \cite{gerchberg1972practical}. These algorithms are conceptually simple, efficient to implement, and, hence, popular among practitioners, despite the lack of rigorous mathematical understanding of their properties and the lack of global recovery guarantees. In particular, delicate parameter and initial guess selection is required to enable reconstruction, or there might be multiple stationary points, or a stationary point outside the non-convex set $\{z: |\mathcal{F}(z)(\omega)| = |\mathcal{F}(x)(\omega)|\}$.\par 


More recently, the problem of developing phase retrieval algorithms which are efficient, have provable recovery guarantees, and are robust to noise in the measurements received a lot of attention.\par

The first phase retrieval algorithm for which recovery and robustness guarantees were proven is \emph{PhaseLift}. It is based on the idea of lifting, proposed by Balan, Bodmann, Casazza, and Edidin \cite{balan2}. They noticed that the intensity measurements can be lifted into the $M^2$-dimensional real vector space $H^M$ of self-adjoint matrices, where the measurements become of the form of Hilbert-Schmidt inner products $b_j=|\langle \varphi_j, x \rangle|^2 = \langle xx^*, \varphi_j \varphi_j^*\rangle_{HS}$. The phase retrieval problem is then translated to a rank minimization problem over $A = \{X\in H^M \text{, s.t. } \langle X, \varphi_j \varphi_j^*\rangle_{HS} ~=~b_j,~ X\succeq~0\}$. Inspired by this lifting trick, Cand\`{e}s, Eldar, Strohmer, and Voroninski used trace minimization over $A$ as a convex relaxation for the rank minimization \cite{candes1, candes2}. The main result of \cite{candes3} states that if the measurement vectors $\varphi_j$ are sufficiently randomized and the number of measurements $N$ is $O(M)$, then the solution of the trace minimization is exact and robust in presence of noise.  Partial derandomizations of PhaseLift have been proposed by Candes, Li, and Soltanolkotabi \cite{candes4} and by Gross, Krahmer, and K\"{u}ng \cite{gross2015partial}.\par

Another phase retrieval algorithm called \emph{Wirtinger flow} has been proposed by Candes, Li, and Soltanolkotabi \cite{candes2015phase}. It is based on non-convex optimization and has two components: a careful initialization obtained by means of spectral method, and a series of updates based on stochastic gradient descent scheme that refine this initial estimate. A modification of this algorithm, \emph{truncated Wirtinger flow}, has been shown to recover signal $x$ from $O(M)$ sufficiently randomized phaseless measurements and is robust in presence of noise \cite{Chen}.

A fast phase retrieval algorithm from \emph{local correlation measurements} is proposed in \cite{iwen2015fast}. In this paper a well-conditioned set of Fourier-based measurements is constructed. These measurements are guaranteed to allow for the phase retrieval of a given vector $x\in \mathbb{C}^M$ with high probability in $O(M \log^4 M)$-time.

Another method called \emph{phase retrieval with polarization} has been proposed by Alexeev, Bandeira, Fickus, and Mixon \cite{mixon1}. The main idea of this method is described in Section \ref{polarization_idea}.

While for all these algorithms recovery guarantees have been established, most of them are designed to work exclusively with randomly generated frames. Since measurements of this type are not implementable in practice, the design of fast and stable recovery algorithms with a small number of application relevant, structured, measurements remains an important problem, which we address in this paper.
}
\subsection{Main result}

We study the phase retrieval problem for \emph{Gabor frame} measurements, that is, the case when frame vectors are given by time and frequency shifts of a randomly chosen vector, called the \emph{Gabor window} (see Section \ref{GaborFr} for a precise definition). The main motivation for using Gabor frames is that Gabor frame coefficients are masked Fourier coefficients, where the masks are time (or space) shifts of the Gabor window. This makes measurements implementable in applications while preserving the flexibility of the frame-theoretic approach.  

Apart from diffraction imaging, phase retrieval with Gabor frames also arises, for example, in speech recognition problems. The use of a noisy phase or its estimation is a critical problem in speech recognition that may preventing the accurate reconstruction of a signal. There is a longstanding belief that speech recognition should be independent of phase.  Balan, Casazza, and Edidin addressed this conjecture by using the phase retrieval framework to show that the reconstruction of a signal is possible without using a phase or its estimation for a generic frame \cite{balan1}. However, construction of Gabor frames with such property still remains an open problem. Previous work on phase retrieval with Gabor frames concentrated on injectivity conditions for full Gabor frames and shows reconstruction from $M^2$ time-frequency structured measurements of an $M$-dimensional signal \cite{bojarovska2016phase}, while no injectivity results for Gabor frames of smaller cardinality are available to this date.

Based on the idea of polarization \cite{mixon1}, we propose a reconstruction algorithm for time-frequency structured measurements and investigate its robustness in the case when measurements are corrupted by noise. More precisely, we consider measurements of the form 
\begin{equation}\label{intro_meas}
|\langle x, \varphi_j \rangle|^2 + \nu_j, ~ \varphi_j\in \Phi,
\end{equation}
\noindent where $\nu_j$ are noise terms and $\Phi = \Phi_V\cup \Phi_E$ is the measurement frame with  a Gabor frame $\Phi_V$ given by \eqref{vertex_frame} and a set of vectors for additional measurements $\Phi_E$ given by \eqref{edge_frame}. We note that, while the frame $\Phi$ used for reconstruction is not a Gabor frame, the set of vectors for the additional measurements $\Phi_E$ also obeys time-frequency structure and measurements of a signal $x$ with respect to $\Phi_E$ have the form of  windowed Fourier transform measurements, see Section \ref{measurement_process} for the details. We prove the following result, a more precise formulation of which we state in Section~\ref{alg_noisy} as Theorem~\ref{stability_main}.

\begin{theorem} \label{stability_main_int}
Fix $x\in \mathbb{C}^M$ and consider the time-frequency structured phaseless measurements given by (\ref{intro_meas}). If the noise vector $\nu$ satisfies $\frac{||\nu||_2}{||x||_2^2}\le \frac{c}{M}$ for some $c$ sufficiently small, then there exists a numerical constant $C$ independent of $M$, so that for the estimate $\tilde{x}$ produced by Algorithm \ref{noisy_reconstruction} the following holds with overwhelming probability
\begin{equation}\label{recovery_guarantees_intro}
\min_{\theta\in [0,2\pi)}||\tilde{x} - e^{i\theta}x||_2^2 \le \frac{C\sqrt{M}||\nu||_2}{\Delta}.
\end{equation}
\noindent Here $\displaystyle \Delta = \min_{\Lambda'\subset \Lambda, |\Lambda'|\ge 2/3|\Lambda|}\sigma_{\min}^2(\Phi_{\Lambda'}^*)$ is the numerically erasure-robust frame bound for the Gabor frame $\Phi_V$.
\end{theorem}

The reconstruction algorithm we propose in this paper requires a close to optimal number $N = O(M\log M)$ of time-frequency structured measurements. To the best of our knowledge, Theorem~\ref{stability_main_int} provides the best existing robustness guarantee for measurements obeying time-frequency structure. Moreover, our result  guarantees exact recovery of $x$ (up to a global phase) in the noiseless case, that is, when $\nu = 0$. We note that the reconstruction guarantees given in Theorem~\ref{stability_main_int} are non-uniform, in the sense that inequality~\eqref{recovery_guarantees_intro} holds with high probability for each particular signal $x$, but not for all $x\in \mathbb{C}^M$ simultaneously. Obtaining a uniform version on Theorem~\ref{stability_main_int} is one of the main directions for future work. Numerical experiments, presented in Section \ref{numerical_robustness}, verify the robustness of the proposed phase retrieval algorithm and illustrate dependencies of the error-to-noise ratio on various parameters.\par

\begin{remark}
The bound obtained in Theorem~\ref{stability_main_int} is similar, up to a $\log$ factor, to the recovery guarantees obtained in \cite{mixon1} for phase retrieval with random Gaussian frames with independent frame vectors. Note that, in equation~\eqref{recovery_guarantees_intro}, the norm of the reconstruction error $\tilde{x} - e^{i\theta}x$ is squared, while the norm of the noise vector $\nu$ is not. This makes the obtained bound somewhat weaker compared to the recovery guarantees for the PhaseLift reconstruction algorithm (for random frames with independent frame vectors), which ensure that, if $||\nu||_2<\epsilon$, the estimate $\tilde{x}$ obtained using PhaseLift satisfies $\min_{\theta\in [0,2\pi)}||\tilde{x} - e^{i\theta}x||_2 \le C\epsilon$, for an appropriately chosen numerical constant~$C>0$~\cite{candes2}. At the same time, while the noise bound $\epsilon$ is an input parameter of the PhaseLift algorithm, the method described in this paper is independent of the actual noise size.
\end{remark}

The remaining part of this paper is organized as follows. In Section \ref{backgr}, we describe the idea of polarization and give some basic definitions and results from Gabor analysis and the theory of expander graphs that are used in the sequel. In Section \ref{algorithm}, we describe the reconstruction algorithm for time-frequency structured measurements and discuss the robustness of this algorithm in Section \ref{Stability}. The analysis of robustness of the constructed algorithm leads to the investigation of geometric properties of Gabor frames, such as order statistics of frame coefficients. These properties are discussed in Section \ref{projective_uniformity}. Numerical results of the algorithm's robustness are presented in Section \ref{numerical_robustness}. 

\section{Notation and setup}\label{backgr}

Here and in the sequel, $\odot$ denotes pointwise multiplication of two vectors of the same dimension. We view a vector $x\in\mathbb{C}^M$ as a function $x:\mathbb{Z}_M\to~\mathbb{C}$, that is, all the operations on indices are done modulo $M$ and $x(m-k) = x(M+m -k)$. We denote the complex unit sphere by $\mathbb{S}^{M-1} = \{x\in \mathbb{C}^M,  ||x||_2 = 1\}$.\par

The adjoint matrix of $A\in \mathbb{C}^{k\times m}$ is denoted by $A^*\in \mathbb{C}^{m\times k}$,  and the smallest singular value of $A$ is denoted by $\sigma_{\min}(A)$.  Also, by a slight abuse of notation, we identify a frame $\Phi = \{\varphi_j\}_{j = 1}^N\subset \mathbb{C}^M$ with its synthesis matrix, having the frame vectors $\varphi_j$ as columns. For any $V\subset \{1,\dots,N\}$, we set $\Phi_V = \{\varphi_j\}_{j \in V}$.\par

We denote the Bernoulli distribution with success probability $p$ by $~B\left(1,p\right)$. Further, $\mathcal{N}(\mu, \sigma)$ denotes the Gaussian distribution with mean $\mu$ and variance~$\sigma^2$, and $\mathbb{C}\mathcal{N}(\mu, \sigma)$ denotes the complex valued Gaussian distribution.

\subsection{Phase retrieval with polarization}\label{polarization_idea}

The polarization approach to phase retrieval can be described as follows~\cite{mixon1}. Suppose $\Phi_V = \{\varphi_j\}_{j\in V}\subset \mathbb{C}^M$ is a measurement frame. We consider the phase retrieval problem
\begin{align*}
     &\text{find } & x\\ \numberthis \label{nophases}
       &\text{subject to } & |\langle x, \varphi_j\rangle|^2 = b_j.
 \end{align*}
For any $(i,j)\in V\times V$ with $|\langle x,\varphi_i\rangle| \ne 0$ and $|\langle x,\varphi_j\rangle| \ne 0$, we define the \emph{relative phase} between frame coefficients as
\begin{equation}\label{RelPh_polar}
\omega_{ij}=\left(\frac{\langle x, \varphi_i \rangle}{|\langle x, \varphi_i\rangle|} \right)^{-1} \frac{\langle x, \varphi_j\rangle}{|\langle x, \varphi_j\rangle|}=\frac{\overline{\langle x, \varphi_i \rangle}\langle x, \varphi_j\rangle}{|\langle x, \varphi_i\rangle||\langle x, \varphi_j\rangle|}.
\end{equation}
\noindent Note that $\omega_{ij}\omega_{jk} = \omega_{ik}$. Suppose that we are given $\{\omega_{ij}\}_{(i,j)\in E}$ for some set $E\subset V\times V$  in addition to the phaseless measurements with respect to $\Phi_V$. Then we seek to solve the simpler problem
\begin{align*}
     &\text{find } & x\\ \numberthis \label{withphases}
       &\text{subject to } & \frac{\overline{\langle x, \varphi_i \rangle}\langle x, \varphi_j \rangle}{|\langle x, \varphi_i\rangle||\langle x, \varphi_j\rangle|} = \omega_{ij}, \\
&& |\langle x, \varphi_i\rangle|^2 = b_i.
 \end{align*}
This problem can be solved using \emph{phase propagation}. More precisely, we choose $|\langle x,\varphi_{i_0}\rangle| \ne 0$, set $c_{i_0} = |\langle x,\varphi_{i_0}\rangle|$, and for every $j\in V$ with $(i_0,j)\in E$ define
\begin{equation*}\label{Rec}
c_{j} =
  \begin{cases}
   \omega_{i_0j} |\langle x, \varphi_j\rangle| & \quad \text{if } |\langle x, \varphi_j\rangle|\ne 0,\\
   0  & \quad \text{otherwise}.\\
  \end{cases}
\end{equation*}
\noindent In the next step, for each $k$ with $c_k$ not defined yet and $(i_0,j), (j,k)\in E$ for some $j$ with $b_j\ne 0$, we set
\begin{equation*}
c_{k} =
  \begin{cases}
   \omega_{jk}\frac{c_j}{|c_j|} |\langle x, \varphi_k\rangle| & \quad \text{if } |\langle x, \varphi_k\rangle|\ne 0,\\
   0  & \quad \text{otherwise}.\\
  \end{cases}
\end{equation*}
\noindent We repeat this step iteratively until values $c_i$ are assigned to all indices $i\in V$ that can be reached from $i_0$ using edges from $E$. This process is illustrated in Figure \ref{propagation} (left).\par 

Assume that we were able to compute $c_i$ for all $i\in V$. Then, using a dual frame $\tilde{\Phi}_V = \{\tilde{\varphi}_i\}_{i\in V}$ and treating $c_i$'s as frame coefficients, we reconstruct a representative of the ``up-to-a-global-phase'' equivalence class $[x]$ as
\begin{equation*}
\sum_{j\in V}c_j\tilde{\varphi}_j=\sum_{j\in V}\omega_{i_0j}|\langle x,\varphi_j\rangle|\tilde{\varphi}_j=\big(\tfrac{\langle x,\varphi_{i_0}\rangle}{|\langle x,\varphi_{i_0}\rangle|}\big)^{-1}\sum_{j\in V}\langle x,\varphi_j\rangle\tilde{\varphi}_j=\big(\tfrac{\langle x,\varphi_{i_0}\rangle}{|\langle x,\varphi_{i_0}\rangle|}\big)^{-1}x\in [x].
\end{equation*}

Let us consider the graph $G = (V,E)$, later called the \emph{graph of measurements}, with the set of vertices indexed by $V$ and the set of edges $E$. From the phase propagation procedure, it is apparent that if $\langle x,\varphi_j \rangle = 0$ for some $j\in V$, then the corresponding relative phases $\omega_{ji}$ are not defined for all $i\in V$ and the phase cannot be propagated through vertex $j$. This has the effect of deleting vertex $j$ from $G$, see Figure \ref{propagation} (right). If $G$ remains connected after deleting all ``zero'' vertices, then, for every vertex $i$, there exists a path from $i_0$ to $i$, and $c_i$ can be computed. This solves problem (\ref{withphases}).\par

Thus, the initial phase retrieval problem (\ref{nophases}) is reduced to the problem of finding relative phases between pairs of frame coefficients from a set $E$, so that the corresponding graph of measurements $G = (V,E)$ satisfies strong connectivity properties. To obtain the relative phase between frame coefficients, the following polarization identity is useful.
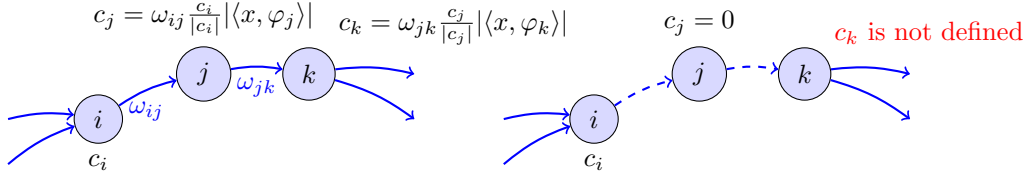
\begin{figure}
\begin{tabular}{cc}
\begin{minipage}[t]{0.45\textwidth}
\begin{tikzpicture} 
\draw (-3*\step + 0.4cm, -  \step ) node(B)[circle,fill=blue!15,draw,,label=below:{\footnotesize $c_{i}$}]{\footnotesize $i$};

\draw (3*\step - 0.4cm, 0 ) node(C)[circle,fill=blue!15,draw, label=above right:{ \footnotesize $c_{k}=\omega_{jk} \frac{c_j}{|c_j|}|\langle x,\varphi_k\rangle|$}]{\footnotesize  $k$};

\draw (0, 0) node(A)[circle,fill=blue!15,draw, label=above:{ \footnotesize $c_{j}=\omega_{ij} \frac{c_i}{|c_i|}|\langle x,\varphi_j\rangle|$}]{\footnotesize $j$};

\draw (A) edge[->, thick, bend left=10, color=blue] node [below] {\footnotesize $\omega_{jk}$} (C);

\draw (B) edge[->, thick, bend left=10, color=blue] node [below] {\footnotesize $\omega_{ij}$} (A);


\draw (-5*\step + 0.4cm, - 2 * \step) edge[->,  thick, bend left=10, color=blue] node [left] {} (B);

\draw (-5*\step + 0.4cm, -  \step) edge[->,  thick, bend left=10, color=blue] node [left] {} (B);
  
\draw (C) edge[->,  thick, bend left=10, color=blue] node [left] {} (4*\step + 0.4cm, - \step);

\draw (C) edge[->,  thick, bend left=10, color=blue] node [left] {} (4*\step + 0.4cm, 0);

\end{tikzpicture}
\end{minipage}
&
\begin{minipage}[t]{0.45\textwidth}
\begin{tikzpicture} 

\draw (-3*\step + 0.4cm, -  \step ) node(B)[circle,fill=blue!15,draw,,label=below:{\footnotesize $c_{i}$}]{\footnotesize $i$};

\draw (3*\step - 0.4cm, 0 ) node(C)[circle,fill=blue!15,draw, label=above right:\textcolor{red}{\footnotesize $c_{k} \text{ is not defined}$}]{\footnotesize $k$};

\draw (0, 0) node(A)[circle,fill=blue!15,draw, label=above:{\footnotesize $c_{j}=0$}]{\footnotesize $j$};

\draw (A) edge[->, thick, bend left=10, color=blue, dashed] node [below] {} (C);

\draw (B) edge[->, thick, bend left=10, color=blue, dashed] node [below] {} (A);


\draw (-5*\step + 0.4cm, - 2 * \step) edge[->,  thick, bend left=10, color=blue] node [left] {} (B);

\draw (-5*\step + 0.4cm, -  \step) edge[->,  thick, bend left=10, color=blue] node [left] {} (B);
  
\draw (C) edge[->,  thick, bend left=10, color=blue] node [left] {} (4*\step + 0.4cm, - \step);

\draw (C) edge[->,  thick, bend left=10, color=blue] node [left] {} (4*\step + 0.4cm, 0); 
\end{tikzpicture}
\end{minipage}

\end{tabular}
\caption{\label{propagation} Phase propagation process described above is shown on the left. We iteratively compute phases of the measurements using relative phases \eqref{RelPh_polar} and phases computed on previous step. On the right, phase propagation through vertex $j$ fails since the corresponding measurement is zero and relative phases $\omega_{ij}$ and $\omega_{jk}$ are not defined.}
\end{figure}

\begin{lemma} \emph{\textbf{\cite{mixon1}}}\label{polarcor}
Let $\omega = e^{2\pi i /3}$. If $\langle x,\varphi_i\rangle \ne 0$ and $\langle x,\varphi_j\rangle \ne 0$, then
\begin{equation*}
\omega_{ij}=\frac{1}{3|\langle x,\varphi_i\rangle||\langle x,\varphi_j\rangle|}\sum_{k=0}^2\omega^{k}\big|\langle x,\varphi_i+\omega^{k}\varphi_j\rangle\big|^2.
\end{equation*}
\end{lemma}

In other words, to compute the relative phase $\omega_{ij}$ between the nonzero frame coefficients $\langle x, \varphi_i \rangle$ and $\langle x, \varphi_j \rangle$, we may use three additional phaseless measurements of $x$ with respect to $\varphi_i + \varphi_j$, $\varphi_i + \omega\varphi_j$, and $\varphi_i + \omega^2\varphi_j$. This means that the reconstruction of $x$ using phase propagation involves only phaseless measurements, namely, phaseless measurements with respect to the union $\Phi_V\cup \Phi_E$, where $\Phi_V$ is a ``vertex'' frame and $\Phi_E = \{\varphi_i + \omega^k\varphi_j\}_{(i,j)\in E,~k\in\{0,1,2\}}$. Note that $|\Phi_V\cup \Phi_E| = |V| + 3|E|$.\par

In \cite{mixon1}, it is shown that in the noiseless case, one can perform phase retrieval with polarization using only $O(M)$ measurements. This algorithm is robust provided $\Phi_V$ consists of independent Gaussian vectors and the number of measurements is $O(M\log{M})$. In \cite{mixon2}, Bandeira, Chen, and Mixon adapt the polarization method to magnitude measurements of masked Fourier transforms of the signal. Using tools from additive combinatorics, the authors show that the graph of measurements they are using for reconstructon is sufficiently connected provided that the total number of measurements is~$O(M\log M)$. However, no stability results are given for the case of structured measurements as considered in \cite{mixon2}.\par

In Section \ref{algorithm} we use the idea of polarization to build a recovery algorithm for time-frequency structured measurements and show reconstruction and stability guarantees for the designed algorithm.

\subsection{Gabor frames for $\mathbb{C}^M$}\label{GaborFr}

Let us begin by defining two families of unitary operators on $\mathbb{C}^M$, namely, cyclic shift operators and modulation operators.
\begin{definition}\mbox{}
\begin{enumerate}[leftmargin=*]
\item \emph{Translation} (or \emph{time shift}) by $k\in \mathbb{Z}_M$, is given by
\begin{equation*}
T_k x = T_k\left( x(0), x(1),\dots, x(M-1)\right) = \left(x(m-k)\right)_{m\in \mathbb{Z}_M}.
\end{equation*}
\noindent That is, $T_k$ simply permutes entries of $x$ using $k$ cyclic shifts.
\item \emph{Modulation} (or \emph{frequency shift}) by $\ell\in \mathbb{Z}_M$ is given by 
\begin{equation*}
M_{\ell} x = M_{\ell}\left( x(0), x(1),\dots, x(M-1)\right) = \left(e^{2\pi i \ell m/M}x(m)\right)_{m\in \mathbb{Z}_M}.
\end{equation*}
\noindent That is, $M_{\ell}$ multiplies $x = x(\cdot)$ pointwise with the harmonic $e^{2\pi i\ell(\cdot) /M}$.

\item The superposition $\pi(k,\ell) = M_{\ell}T_k$ of translation by $k$ and modulation by $\ell$ is a \emph{time-frequency shift operator}.

\item For $g\in\mathbb{C}^M\setminus \{0\}$ and $\Lambda\subset  \mathbb{Z}_M \times \mathbb{Z}_M$, the set of vectors 
\begin{equation*}
(g, \Lambda) = \{\pi (k, \ell)g\}_{(k,\ell)\in \Lambda}
\end{equation*}
\noindent is called the \emph{Gabor system} generated by the \emph{window} $g$ and the set $\Lambda$. A Gabor system  which spans $\mathbb{C}^M$ is a frame and is referred to as a \emph{Gabor frame}.
\end{enumerate}
\end{definition}

The \emph{discrete Fourier transform} $\mathcal{F}:\mathbb{C}^M\to \mathbb{C}^M$ plays a fundamental role in Gabor analysis. It is given pointwise by 
\begin{equation*}\label{DFT}
\mathcal{F}x(\ell) = \sum_{m\in\mathbb{Z}_M}{x(m)e^{-2\pi i m \ell /M}}, ~ \ell\in\mathbb{Z}_M.
\end{equation*}
\indent The \emph{short-time Fourier transform} (or \emph{windowed Fourier transform}) \linebreak$V_g : \mathbb{C}^M\to\mathbb{C}^{M\times M}$ with respect to the window $g\in\mathbb{C}^M\setminus \{0\}$ is given by
\begin{equation}\label{STFT}
(V_g x)(k,\ell) = \langle x, \pi(k,\ell)g\rangle = \mathcal{F}(x\odot T_k\bar{g})(\ell),~k,\ell\in \mathbb{Z}_M.
\end{equation}
\noindent Equality (\ref{STFT}) indicates that the short-time Fourier transform on $\mathbb{C}^M$ can be efficiently computed using the \emph{fast Fourier transform} (FFT), an efficient algorithm to compute the discrete Fourier transform of a vector. Phase retrieval with time-frequency structured measurements benefits from this, as it reduces the run time of recovery algorithms.\par 

As we shall use polarization for phase retrieval, we would like to choose a window function $g$ so that the frame $(g, \Lambda)$ is a \emph{full spark frame}, that is, so that for any subset $S\subset(g, \Lambda)$ of frame vectors with $|S|\ge M$, $S$ spans $\mathbb{C}^M$ \cite{Alex}. Note that if the full Gabor system $(g, \mathbb{Z}_M\times \mathbb{Z}_M)$ is full spark, then so is $(g, \Lambda)$, for any $\Lambda \subset \mathbb{Z}_M\times \mathbb{Z}_M$.\par

The following result on the spark of Gabor frames with random window was shown for $M$ prime in \cite{pfander1} and for $M$ composite in \cite{malik}.
\begin{theorem}\label{FullSp}
Let $M$ be a positive integer and let $\Lambda$ be a subset of \linebreak$\mathbb{Z}_M \times \mathbb{Z}_M$ with $|\Lambda|\ge M$. Then, for almost all windows $\zeta$ on the complex unit sphere $\mathbb{S}^{M-1}\subset \mathbb{C}^M$, $(\zeta, \Lambda)$ is a full spark frame.
\end{theorem}

A more detailed description of Gabor frames in finite dimensions and their properties can be found in \cite{pfander2}.

\section{Phase retrieval from Gabor measurements}\label{algorithm}

We now describe our  design of a measurement frame and a reconstruction process in the noisless case, also addressed in \cite{salanevich}, and then discuss the robustness of our algorithm in the case when measurements are corrupted by noise. The analysis of the algorithm's robustness leads to the investigation of geometric properties of the measurement frame, such as frame bounds and flatness of the vector of frame coefficients. These properties are not only important for the problem at hand, but are of general interest in Gabor analysis.

\subsection{Measurement process and frame construction}\label{measurement_process}

Consider the phase retrieval problem (\ref{nophases}) with measurement frame 
\begin{equation*}
\Phi = \{\varphi_j\}_{j = 1}^N = \Phi_V\cup \Phi_E \subset\mathbb{C}^M,
\end{equation*}
\noindent where $\Phi_V$ is a Gabor frame and $\Phi_E$ is a set of vectors corresponding to the additional edge measurements. As described in Section \ref{polarization_idea}, phaseless measurements with respect to $\Phi_E$ are used to compute relative phases between frame coefficients.\par

Let us specify $\Phi_V$ and $\Phi_E$ now. For $\Phi_V$ we choose the Gabor frame
\begin{equation}\label{vertex_frame}
\begin{split}
& \Phi_V = (g,\Lambda) \text{ with } \Lambda = F\times \mathbb{Z}_M, ~ F\subset\mathbb{Z}_M, ~ |F|=K,\\
\text{ and } g & \in \mathbb{C}^M \text{ uniformly distributed on the unit sphere } \mathbb{S}^{M-1}\subset \mathbb{C}^M.
\end{split}
\end{equation}
The integer $K$ is fixed and does not depend on the ambient dimension $M$, and $F$ is an arbitrary subset of $\mathbb{Z}_M$ of cardinality $K$. That is, we consider all frequency shifts and only a constant number of time shifts. As equation~(\ref{STFT}) indicates, our measurements are magnitudes of masked Fourier transform coefficients with the masks being $T_k\bar{g}$, $k\in F$.\par

We choose
\begin{equation}\label{edge_frame}
\begin{split}
\Phi_E = & \left\lbrace\pi(\lambda_1)g + \omega^t\pi(\lambda_2)g \right\rbrace_{(\lambda_1,\lambda_2)\in E,~ t\in \{0,1,2\}} \text{ with }\omega = e^{2\pi i/3},\\
E = & \left\lbrace((k_1,\ell_1),(k_2,\ell_2)) \text{, s.t. } k_1,k_2\in F, ~\ell_2 - \ell_1\in C\right\rbrace\subset \Lambda\times\Lambda,\\
\text{and } C = & D\cup (-D)\setminus \{0\}\subset \mathbb{Z}_M \text{ with } {\bf 1}_D(m)\sim \text{i.i.d. } B\left(1,\tfrac{d\log M}{M}\right).
\end{split}
\end{equation}
\noindent In other words, $D\subset\mathbb{Z}_M$ is constructed at random, so that every $m\in \mathbb{Z}_M$ is chosen to be an element of $D$ independently with probability $\tfrac{d\log M}{M}$, for a parameter~$d>0$ we will specify later, and $C = D\cup (-D)\setminus \{0\}$. Then
\begin{equation*}
 \pi(\lambda_1)g+ \omega^t \pi(\lambda_2) g  =\pi(k_1,\ell_1)g+\omega^t\pi(k_2, \ell_2)g = p_{(\ell_2 - \ell_1) k_1 k_2}(t) \odot \pi(\lambda_1)g,
\end{equation*}
\noindent where the vector $p_{c k_1 k_2}(t)\in\mathbb{C}^M$ is defined pointwise by
\begin{equation*}
p_{c, k_1, k_2}(t)(m) =1 + e^{2\pi i \left( \frac{cm}{M} + \frac{t}{3}\right)}\,\frac{g(m-k_2)}{g(m-k_1)},\quad m\in \mathbb{Z}_M,
\end{equation*}
\noindent with parameters $c\in C$, $k_1, k_2\in F$, and $t\in \{0,1,2\}$. Therefore, for each fixed set of four parameters $(c, k_1, k_2, t)$, the respective additional measurements are magnitudes of masked Fourier transform coefficients as well, namely,
\begin{equation}\label{Add}
\left | \langle x, p_{c k_1 k_2}(t) \odot  \pi(k_1,\ell)g \rangle \right | = \left| \mathcal{F}\left(x \odot \bar{p}_{c k_1 k_2}(t) \odot T_{k_1} \bar{g} \right) (\ell)  \right|, \quad \ell\in \mathbb{Z}_M.
\end{equation}

Let us note that the frame $\Phi = \Phi_V\cup\Phi_E$ constructed in this way consists of $|\Lambda| + 3|E| = |F|M + 3|F|^2\,|C|\,M$ vectors. Since $C = D\cup (-D)\setminus \{0\}$ with ${\bf 1}_D(m)\sim \text{i.i.d. } B\left(1,\frac{d\log M}{M}\right)$, we have $|C| = O(\log M)$ with high probability and thus $|\Phi| = O(M\log M)$.

Using the polarization identity in Lemma \ref{polarcor} with $\omega = e^{2\pi i/3}$, we compute the relative phases
\begin{equation}\label{RelPh}
\omega_{\lambda_1 \lambda_2}=\left(\tfrac{\langle x, \pi(\lambda_1)g \rangle}{|\langle x, \pi(\lambda_1)g\rangle|} \right)^{-1} \tfrac{\langle x, \pi(\lambda_2)g\rangle}{|\langle x, \pi(\lambda_2)g\rangle|} = \frac{\sum_{t = 0}^2\omega^t|\langle x, \pi(\lambda_1)g + \omega^t \pi(\lambda_2)g \rangle|^2}{3|\langle x, \pi(\lambda_1)g\rangle||\langle x, \pi(\lambda_2)g\rangle|},
\end{equation}
\noindent for $(\lambda_1,\lambda_2)\in E$, where $E$ is defined by (\ref{edge_frame}). Recall, that $\omega_{\lambda_1 \lambda_2}$ is well defined if and only if $|\langle x,\pi(\lambda_1)g\rangle| \ne 0$ and $|\langle x,\pi(\lambda_2)g\rangle| \ne 0$.

\begin{remark}
As equations (\ref{STFT}) and (\ref{Add}) indicate, all required measurements are magnitudes of masked Fourier transform coefficients. These are relevant for many applications. Moreover, measurements and reconstruction in this case can be implemented using FFT, which allows a noticeable speed up of measurement and reconstruction processes. For comparison, the computational complexity of the measurement process with random Gaussian frame of cardinality $O(M\log M)$ (as considered, for example, in \cite{candes2} and \cite{mixon1}) is $O(M^2\log M)$, and the complexity of measurement with the frame $\Phi$ constructed above is $O(M\log^2 M)$. Furthermore, in the case of random Gaussian frames, we have to use $O(M^2\log M)$ memory bits to store the measurement matrix, while for our frame $\Phi$ it is enough to store the window $g$ and the set $C$, and the overall amount of memory used is only $M + O(\log M) = O(M)$. These are some of the advantages of time-frequency structured frames in comparison to randomly generated frames.
\end{remark}

\subsection{Reconstruction in the noiseless case}

We now describe our polarization based reconstruction algorithm for time-frequency structured measurements.\par
Let us consider graph $G = (\Lambda, E)$, where $\Lambda$ and $E$ are defined by equations \eqref{vertex_frame} and \eqref{edge_frame}, respectively. Since $0\notin C$, the graph $G$ has no loops, and since $C=-C$, it is not directed. Also, each vertex $\lambda = (k,\ell)$ of $G$ is adjacent to any vertex $\lambda' = (k', \ell + c)$ with $c\in C$ and $k'\in F$. Thus, each vertex in $G$ has degree $|F||C|$ and $G$ is regular.\par
%

\begin{figure}[t]\center
\begin{tabular}{ccc}
\includegraphics[width=0.31\linewidth]{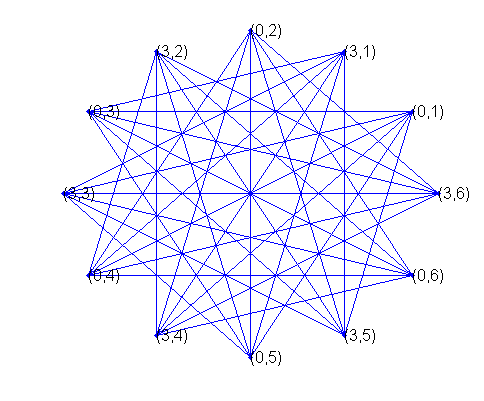}
&
\includegraphics[width=0.31\linewidth]{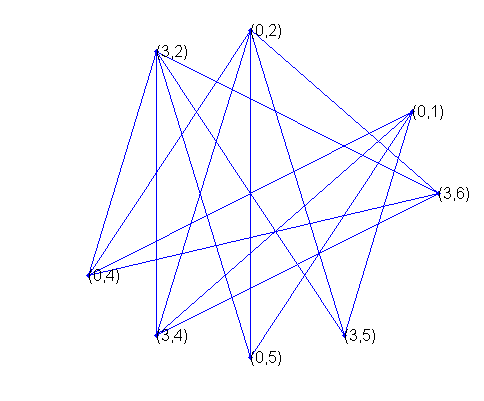}
&
\includegraphics[width=0.31\linewidth]{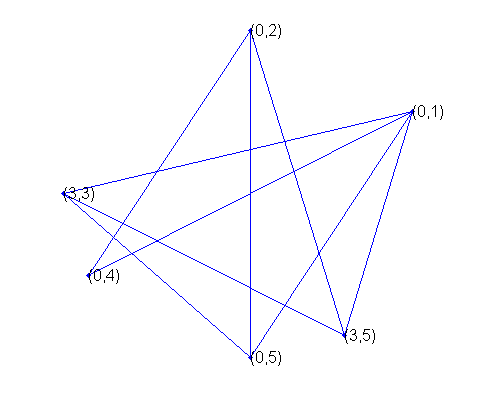}
\end{tabular}
\caption{\label{graph} An example of the graph of measurements $G$ with $M = 6$, $F = \{0,3\}$ and $C = \{2,3,4\}$ (left). This graph remains connected after deleting one third of its vertices (middle). After we delete one half of its vertices it has a connected component of size at least 4 (right).}
\end{figure}


Let $A = A(G)$ be the \emph{adjacency matrix} of $G$, that is, a $|\Lambda|\times |\Lambda|$ matrix whose $(\lambda_1,\lambda_2)$ entry is equal to the number of edges in $G$ connecting vertices $\lambda_1$ and $\lambda_2$. Being real and symmetric, $A(G)$ has $|\Lambda|$ real eigenvalues ${\alpha_1\ge \alpha_2\ge \dots\ge \alpha_{|\Lambda|}}$. We refer to the sequence $\{\alpha_i\}_{i=1}^n$ of eigenvalues  of $A(G)$ as the \emph{spectrum} of the graph $G$. Note that for a $d$-regular graph $G$, $\alpha_1 = d$. The spectrum encodes information about the connectivity of the graph. For example, $G$ is connected if and only if $\alpha_1 > \alpha_2$. 
The value $\Spg (G) = \frac{\alpha_1-\alpha}{\alpha_1}$, where $\alpha = \max\{|\alpha_2|, |\alpha_n|\}$, is known as the \emph{spectral gap} of $G$. The following version of \cite[Lemma 5.2.]{harsha} relates connectivity properties of the graph $G$ and its spectral gap, see also \cite{mixon1}.

\begin{lemma}\label{Spg_cor}
Let $G$ be a $d$-regular graph. For all $\varepsilon \le \frac{\Spg (G)}{6}$, the graph obtained by removing any $\varepsilon n$ vertices from $G$ has a connected component of size at least $\left(1 - \frac{2\varepsilon}{\Spg (G)} \right)n$.
\end{lemma}

\begin{algorithm}
    \SetKwInOut{Input}{Input}
    \SetKwInOut{Output}{Output}

    \Input{phaseless measurements $b$ with respect to $\Phi_\Lambda\cup\Phi_E$, defined by (\ref{vertex_frame}) and (\ref{edge_frame}); $F\subset\mathbb{Z}_M$, $C\subset \mathbb{Z}_M$, and window $g\in \mathbb{C}^M\setminus\{0\}$}
    \Output{$\tilde{x}\in [x]$, i.e. the signal $x$ up to a global phase.}

construct the graph $G = (\Lambda,E)$ with $\Lambda = F\times \mathbb{Z}_M$ and $E$ as in (\ref{edge_frame})\;
assign to each $\lambda\in \Lambda$ the weight $b_\lambda = |\langle x, \pi(\lambda)g \rangle|^2$\;
assign to each edge $(\lambda_1, \lambda_2)\in E$ the weight $\omega_{\lambda_1 \lambda_2}$ computed using (\ref{RelPh})\;
delete from $G$ all vertices $\lambda$ with $b_\lambda = 0$ to obtain $G' = (\Lambda',E')\subset G$\;
choose a connected component $G'' = (\Lambda'',E'')\subset G'$ of the biggest size\;
run the \emph{phase propagation} process (Section \ref{polarization_idea}) to obtain $c_\lambda$, $\lambda \in \Lambda''$\;
reconstruct $\tilde{x} = (\Phi_{\Lambda''} \Phi_{\Lambda''}^*)^{-1}\Phi_{\Lambda''}c$ from $c = \{c_\lambda\}_{\lambda\in \Lambda''}$.
    \caption{\label{noiseless_case_alg} Reconstruction in the noiseless case}
\end{algorithm}

To be able to reconstruct a signal $x$ using Algorithm \ref{noiseless_case_alg}, we need to ensure that $|\Lambda''|$ is sufficiently large, so that $\Phi_{\Lambda''} = (g, \Lambda'')$ is a frame. Then $x$ can be recovered from its frame coefficients with respect to $\Phi_{\Lambda''}$. In terms of the graph of measurements $G$, this means that after we delete all vertices $\lambda$ with zero weight, the resulting graph has a connected component of sufficiently big size. As Lemma \ref{Spg_cor} shows, this is satisfied provided that $G$ has a sufficiently big spectral gap. The spectral gap of $G = (\Lambda, E)$ can be estimated in terms of the set $C$. 
More precisely, for the random set $C$ constructed in \eqref{edge_frame}, the following result is shown in \cite{mixon2}.
\begin{lemma}\label{bernoulli}
Pick $d > 36$ and suppose the entries of the characteristic vector~${\bf 1}_D$ of a set~$D$ are independent with distribution  $B\left(1,\frac{d\log M}{M}\right)$. Take ${C = D\cup (-D)\setminus \{0\}}$ and construct the graph $G$ as above. Then with overwhelming probability
\begin{equation*}
\Spg (G) \ge 1 - \frac{6}{\sqrt{d}}.
\end{equation*}
\end{lemma}
\longhide{\begin{proof}
By , we have $\Spg (G) = 1 - \frac{M}{|C|}||C||_u$. Thus, to conclude the proof, it is enough to show that 
\begin{equation*}
||C||_u\le \frac{6|C|}{\sqrt{d} M}.
\end{equation*}
\end{proof}
}
Now we are ready to establish recovery guarantees for Algorithm \ref{noiseless_case_alg}.

\begin{theorem}\label{noiseless_reconstruction_algorithm}
Let frames $\Phi _V$ and $\Phi_E$ be constructed as above, with \linebreak$|F|=12$ and $d = 144$. Then every signal $x\in \mathbb{C}^M$ can be reconstructed form $M + 3|F|^2M|C| = O(M\log M)$ phaseless measurements with respect to the frame $\Phi_V\cup\Phi_E$ using Algorithm \ref{noiseless_case_alg}.
\end{theorem}

\begin{proof}
We begin the reconstruction algorithm with assigning to each vertex $\lambda\in \Lambda$ of the constructed graph $G$ the weight $b_\lambda = |\langle x,\pi(\lambda)g \rangle|$ and assigning to each edge $(\lambda_1,\lambda_2)\in E$ of $G$ the relative phase $\omega_{\lambda_1\lambda_2}$ which is computed from the additional edge measurements. Theorem \ref{FullSp} implies that the frame $\Phi_V = \{\pi (\lambda)g\}_{\lambda \in \Lambda}$ is full spark with probability $1$. Thus, for any vector $x\in \mathbb{C}^M$, the number of zero measurements among $\{b_\lambda = | \langle x, \pi(\lambda)g\rangle|\}_{\lambda\in \Lambda}$ is at most $M-1$. In other words, Algorithm \ref{noiseless_case_alg} deletes at most $M-1$ vertices from $G$ to obtain $G'$.\par

Next, $\Phi_V$ being full spark implies that any its subset $\Phi_{V'}\subset\Phi_V$ of size $|\Phi_{V'}|\ge M$ form a frame. Thus, to recover $x$, it is enough to know any $M$ of the frame coefficients with respect to $\Phi_V$. To show that $G'$ has a connected component $G''$ of size at least $M$, first note that Lemma \ref{bernoulli} ensures that $\Spg (G) \ge 1 - \frac{6}{\sqrt{d}} = \frac{1}{2}$. Then, applying Lemma \ref{Spg_cor} with \linebreak$n = |\Lambda| = |F|M$ and $\varepsilon = \frac{1}{|F|} = \frac{1}{12}\le \frac{\Spg G}{6}$, we obtain that after deleting any $\varepsilon n = M$ vertices from $G$, the largest connected component $G''$  will have at least $\left(1-\frac{2\varepsilon}{\Spg G}\right)n\ge \frac{2}{3}|F|M = 8M > M$ vertices.\par

By running the phase propagation algorithm on the connected graph $G'' = (\Lambda'', E'')$, we recover $c_\lambda$, $\lambda\in \Lambda''$, which are, up to a global phase $e^{i\theta}$, $\theta \in [0,2\pi)$, equal to the corresponding frame coefficients  of $x$ with respect to $\Phi_{\Lambda''}$, that is
\begin{equation*}
c_\lambda = e^{i\theta} \langle x, \pi(\lambda)g\rangle, ~ \lambda\in \Lambda''.
\end{equation*}
Now, using the canonical dual frame, we obtain $$\tilde{x} = (\Phi_{\Lambda''} \Phi_{\Lambda''}^*)^{-1}\Phi_{\Lambda''}c = e^{i\theta}(\Phi_{\Lambda''} \Phi_{\Lambda''}^*)^{-1}\Phi_{\Lambda''} \Phi_{\Lambda''}^*x = e^{i\theta} x.$$
\end{proof}

\longhide{Theorem \ref{noiseless_reconstruction_algorithm} shows that $O(M\log M)$ masks suffice for the reconstruction. Unfortunately, the log factor is not just a side effect of the methods we are using for the reconstruction. The following result shows that the $\log$-factor is necessary for polarization-based recovery from masked Fourier coefficients based on the graph $G$ constructed as above \cite{mixon2}.

\begin{theorem}\label{logf}
For $G$ defined as above, $\Spg (G)> \varepsilon$ only if 
\begin{equation*}
|C|\ge \frac{\log M}{2 + \log (1/\varepsilon)}.
\end{equation*}
\end{theorem}
\par
}
\section{Robustness of reconstruction in the presence of noise}\label{Stability}
In many applications, measurements are corrupted by noise. In this section we address the behaviour of the presented algorithm in the case when the available measurements are of the form
\begin{equation}\label{measurements}
\begin{split}
b_\lambda = |\langle x, \pi(\lambda)g \rangle|^2 + \nu_\lambda, & \quad \lambda\in \Lambda;\\
b_{\lambda_1\lambda_2t} = |\langle x, \pi(\lambda_1)g + \omega^t \pi(\lambda_2)g \rangle|^2 + \nu_{\lambda_1\lambda_2t}, & \quad (\lambda_1,\lambda_2)\in E,~ t\in \{0,1,2\},
\end{split}
\end{equation}
\noindent where $\nu_\lambda ,~ \nu_{\lambda_1\lambda_2t}$ are noise terms. We aim to construct a modification of Algorithm \ref{noiseless_case_alg} which, in presence of noise, recovers a close estimate $\tilde{x}$ of the original signal $x$.\par 

\subsection{Order statistics of frame coefficients}\label{projective_uniformity}

To compute a relative phase between two frame coefficients, we rely on formula (\ref{RelPh}). The calculations include division by $|\langle x, \pi (\lambda_1)g\rangle|$ and $|\langle x, \pi (\lambda_2)g\rangle|$ and are therefore very sensitive to perturbations when $|\langle x, \pi (\lambda_1)g\rangle|$ or \linebreak$|\langle x, \pi (\lambda_2)g\rangle|$ is small. While ``zero'' vertices provide no relative phase information, vertices with small vertex measurements lead to unreliable relative phase estimations and should therefore be deleted from the graph. As we require the graph of measurements to have a connected component of size at least $M$ after deleting vertices with small weights, the number of such vertices has to be estimated. To do so, we show that the vector of frame coefficients of a fixed $x\in \mathbb{S}^{M-1}$ with respect to a Gabor frame with random window is ``flat'' with high probability. That is, most of the frame coefficients are in the range of $\frac{c}{\sqrt{M}}$ to $\frac{K}{\sqrt{M}}$, for some suitably chosen constants $K>c>0$. 

\begin{theorem}\label{number_of_small_meas}
Fix $x\in \mathbb{S}^{M-1}\subset \mathbb{C}^M$ and consider a Gabor frame $(g, \Lambda)$ with $\Lambda\subset\mathbb{Z}_M\times \mathbb{Z}_M$ and a random window $g$ uniformly distributed on the unit sphere $\mathbb{S}^{M-1}$. Then the following holds.
\begin{enumerate}
\item[(a)] For any $c>0$ and $k>0$, with probability at least $1-\frac{1}{k^2}$, we have
\begin{equation*}
\left|\left\lbrace\lambda\in \Lambda \text{, s.t. } |\langle x, \pi(\lambda)g \rangle|< \frac{c}{\sqrt{M}}\right\rbrace\right|<|\Lambda|(c^2 + kc).
\end{equation*}
\item[(b)] For any $K>0$ and $k>0$, with probability at least $1-\frac{1}{k^2}$, we have
\begin{equation*}
\left|\left\lbrace\lambda\in \Lambda \text{, s.t. } |\langle x, \pi(\lambda)g \rangle|> \frac{K}{\sqrt{M}}\right\rbrace\right|<|\Lambda|\left(\frac{8}{\pi}e^{-K^2} + k\frac{2\sqrt{2}}{\sqrt{\pi}}e^{-\frac{K^2}{2}}\right).
\end{equation*}
\end{enumerate}
\end{theorem}

The proof of this result is presented in the appendix.

\begin{remark}
A similar result can be shown for a Gabor frame with a window whose entries are independent Gaussian random variables. The proof in this case involves the same steps as the proof of Theorem \ref{number_of_small_meas}. 
\end{remark}

Theorem \ref{number_of_small_meas} is a non-uniform result in the sense that the proven bounds hold with high probability for each \emph{individual} $x$. Note that this does not imply that the same bounds will hold \emph{simultaneously for all} $x\in \mathbb{S}^{M-1}$ with high probability. We give a uniform bound  for the number of large frame coefficients in the following result, see also \cite{salanevich2017geometric}.

\begin{theorem}\label{number_of_big_meas_uniform}
Consider a Gabor frame $(g,\mathbb{Z}_M\times\mathbb{Z}_M)$ with window $g$ whose entries are independent Gaussian random variables with zero mean and variance $\frac{1}{\sqrt{M}}$. Then, for some suitably chosen numerical constants $c, c_1 >0$,
\begin{equation*}
\left|\left\lbrace\lambda\in \mathbb{Z}_M\times\mathbb{Z}_M \text{, s.t. } |\langle x, \pi(\lambda)g \rangle|> \sqrt{\frac{3}{2c}}\frac{\log^2M}{\sqrt{M}}\right\rbrace\right| < \frac{cM}{\log^4 M},
\end{equation*}
\noindent for all $x\in \mathbb{S}^{M-1}$, with probability at least $1 - e^{-c_1\log^3 M}$.
\end{theorem}

\begin{proof}
Let $g$ be a random vector with $g(m)\sim\text{i.i.d.}~\mathcal{CN}\left(0,\frac{1}{\sqrt{M}}\right)$, $m\in\mathbb{Z}_M$, and let $\Phi = \Phi_{\mathbb{Z}_M\times \mathbb{Z}_M}$ be an $M\times M^2$ matrix whose columns are $\pi(\lambda)g$, $\lambda\in \mathbb{Z}_M\times \mathbb{Z}_M$. Fix $s\in \{1,\dots, M^2\}$ and for any $x\in \mathbb{S}^{M-1}$ denote by $S_x$ the set of $\lambda \in \mathbb{Z}_M\times \mathbb{Z}_M$ corresponding to the $s$ biggest in modulus frame coefficients of $x$ with respect to the Gabor frame $(g,\mathbb{Z}_M\times\mathbb{Z}_M)$. Then, for the phase vector $v_x\in \mathbb{C}^{M^2}$ defined by 
\begin{equation*}
v_x(\lambda)  = \left\lbrace\begin{array}{ll}
\frac{\langle x, \pi(\lambda)g\rangle}{|\langle x, \pi(\lambda)g\rangle|}, & \lambda\in S_x\\
0, & \text{otherwise},
\end{array}\right.
\end{equation*}
\noindent we have
\begin{equation*}
x^*\Phi v_x = \sum_{\lambda\in S_x}|\langle x, \pi(\lambda)g\rangle|.
\end{equation*}
Applying the Cauchy–Schwarz inequality to $x^*\Phi v_x = \langle \Phi v_x, x \rangle$, we obtain
\begin{equation*}
\sum_{\lambda\in S_x}|\langle x, \pi(\lambda)g\rangle|\le ||x||_2||\Phi v_x||_2 = ||\Phi v_x||_2.
\end{equation*}
\noindent Note that $v_x$ is an $s$-sparse vector with $||v_x||_2 = \sqrt{s}$. Then, if $s = \frac{cM}{\log^4 M}$ for a suitably chosen numerical constant $c>0$, we have 
\begin{equation*}
\frac{1}{2}||v||_2^2\le ||\Phi v||_2^2\le \frac{3}{2}||v||^2_2,
\end{equation*}
\noindent for any $s$-sparse vector $v\in \mathbb{C}^{M^2}$ with probability at least $1 - e^{-c_1\log^3 M}$, where $c_1>0$ depends only on $c$, see \cite[Theorem~5.1]{krahmer2014suprema}. Thus
\begin{equation*}
\sum_{\lambda\in S_x}|\langle x, \pi(\lambda)g\rangle|\le ||\Phi v_x||_2\le \sqrt{\frac{3s}{2}}
\end{equation*}
\noindent with probability at least $1 - e^{-c_1\log^3 M}$. It follows that with the same probability,
\begin{equation*}
\min_{\lambda\in S_x} |\langle x, \pi(\lambda)g\rangle| \le \sqrt{\frac{3}{2s}} = \sqrt{\frac{3}{2c}}\frac{\log^2M}{\sqrt{M}}.
\end{equation*}
In other words, with probability at least $1 - e^{-c_1\log^3 M}$, for any $x\in \mathbb{S}^{M-1}$\linebreak all except at most $\frac{cM}{\log^4 M} - 1$ frame coefficients are in modulus \linebreak bigger then~$\sqrt{\frac{3}{2c}}\frac{\log^2M}{\sqrt{M}}$.
\end{proof}

Since Theorem \ref{number_of_big_meas_uniform} holds for the full Gabor frame $(g,\mathbb{Z}_M\times\mathbb{Z}_M)$, it also holds for any its subframe $(g, \Lambda)$, $\Lambda\subset \mathbb{Z}_M$. Also, note that while Theorem \ref{number_of_big_meas_uniform} gives a better bound on the number of large frame coefficients than Theorem \ref{number_of_small_meas} (b), it gives a slightly weaker bound on the modulus of the remaining coefficients, namely, $\frac{C\log^2 M}{\sqrt{M}}$ instead of $\frac{C}{\sqrt{M}}$.

\begin{remark}
Note that Theorem 5.1 in \cite{krahmer2014suprema} is formulated for a window $g$ with independent mean-zero, variance one, \emph{$L$-subgaussian} entries. Thus, Theorem \ref{number_of_big_meas_uniform} is also true in this more general case.
\end{remark}

\subsection{Reconstruction from noisy measurements}\label{alg_noisy}
To obtain a modification to Algorithm \ref{noiseless_case_alg} that leads to robust reconstruction, we may assume, without loss of generality, that the signal $x$ lies on the complex unit sphere $\mathbb{S}^{M-1}\subset \mathbb{C}^M$. As mentioned before, instead of deleting vertices with zero weight, a portion of vertices with small weights should be deleted in the first step of the reconstruction algorithm. As shown later, very large measurements can also prevent stable reconstruction, so we delete respective vertices as well. To delete vertices from the graph of measurements, we use Algorithm \ref{deleting_small_and_large}.\\

\begin{algorithm}[H]
    \SetKwInOut{Input}{Input}
    \SetKwInOut{Output}{Output}

    \Input{graph $G = (\Lambda, E)$ with weighted vertices $V$, parameters $\alpha$, $\beta$}
    \Output{graph $G'$ with more ``flat'' vertex weights}
   \For{$i = 0$ to $(1 - \alpha)|\Lambda|$}
      {
        find $\lambda\in \Lambda$ with the smallest value of $b_\lambda$ and delete it from $G$\;
      }
      \For{$j = 0$ to $(1 - \beta)|\Lambda|$}
      {
        find $\lambda\in \Lambda$ with the largest value of $b_\lambda$ and delete it from $G$.
      }
    \caption{\label{deleting_small_and_large} deleting ``small'' and ``large'' vertices}
\end{algorithm}


Let $G'' = (\Lambda'', E'')$ be a subgraph of $G'$ and let $A$ be the weighted adjacency matrix of the graph $G''$ given by
\begin{equation}\label{weighred_adj_matr_noisy}
A(\lambda_1,\lambda_2) = 
\left\{ 
\begin{array}{cl}
\dfrac{\overline{\langle x, \pi (\lambda_1)g\rangle} \langle x, \pi (\lambda_2)g\rangle + \varepsilon_{\lambda_1\lambda_2}}{\left|\overline{\langle x, \pi (\lambda_1)g\rangle} \langle x, \pi (\lambda_2)g\rangle + \varepsilon_{\lambda_1\lambda_2}\right|}, & (\lambda_1,\lambda_2)\in E''   \\
0,  & (\lambda_1,\lambda_2)\notin E'',
\end{array}
\right.
\end{equation}
\noindent where $\lambda_1,\lambda_2\in \Lambda''$ and $\varepsilon_{\lambda_1 \lambda_2} = \frac{1}{3}\sum_{t = 0}^2\omega^t\nu_{\lambda_1\lambda_2 t}$. Then $A(\lambda_i,\lambda_j)$ can be considered as an approximation of the relative phase between frame coefficients corresponding to $\Lambda''$.

Noise might accumulate while passing from one vertex to another in the phase propagation process. Thus, we seek an efficient method to reconstruct the phases of the vertex frame coefficients using measured relative phases \eqref{weighred_adj_matr_noisy}. For this purpose, we shall use the \emph{angular synchronization algorithm}~\cite{singer, mixon1}.
%
\longhide{
Let $\upsilon_{\lambda_i} = \frac{\langle x, \pi (\lambda_i)g\rangle}{|\langle x, \pi (\lambda_i)g\rangle|}$, $\lambda_i\in \Lambda''$, be the true phases of the frame coefficients. Then $A(\lambda_i,\lambda_j)$ can be considered as an approximation of the relative phase $\omega_{\lambda_i \lambda_j} = \upsilon_{\lambda_i}^{-1}\upsilon_{\lambda_j}$. The vector \linebreak$\tilde{\upsilon} = \{\tilde{\upsilon}_{\lambda}\}_{\lambda\in \Lambda}$ given by
\begin{equation}\label{minimize}
\begin{split}
& \tilde{\upsilon} = \arg\min_{\substack{|u_\lambda| = 1,\\ \lambda\in \Lambda''}}\sum_{\{\lambda_i, \lambda_j\} \in E''}{|u_{\lambda_j} - A(\lambda_i, \lambda_j)u_{\lambda_i}|^2} = \arg\min_{\substack{|u_\lambda| = 1,\\ \lambda\in \Lambda''}} u^*(D - \bar{A})u,
\end{split}
\end{equation}
\noindent where $D$ is the diagonal matrix of vertex degrees and $\bar{A}$ is a componentwise conjugate of $A$. The vector $\tilde{\upsilon}$ is an approximation of $\upsilon$. Note that we assume  $|u_{\lambda}| = 1$, $\lambda\in \Lambda''$, and thus the quantity $u^*Du = \sum_{\lambda\in \Lambda''}{d_{\lambda}|u_{\lambda}|^2} = \sum_{\lambda\in \Lambda''}{d_{\lambda}}$ does not vary with $u$. Minimizing the above quantity is equivalent to minimizing
\begin{equation*}
\frac{u^*(D - \bar{A})u}{u^*Du} = \frac{(D^{1/2}u)^*(I - D^{-1/2}\bar{A}D^{-1/2})(D^{1/2}u)}{||D^{1/2}u||_2^2},
\end{equation*}
\noindent which is larger than or equal to the smallest eigenvalue $\alpha_1$ of the \emph{connection Laplacian} matrix $L_1 = I - D^{-1/2}\bar{A}D^{-1/2}$. Equality is achieved if $D^{1/2}u$ is a corresponding eigenvector.\par 

Note that in the case when $G''$ is connected and no noise is present, (\ref{minimize}) has a unique (up to a global phase) minimizer. 
}
The following result shows robustness of the angular synchronization algorithm in the presence of noise~ \cite{mixon1, bandeira}.

\begin{theorem}\label{angular_synchronization}
Consider a graph $G = (\Lambda'', E'')$ with spectral gap $\tau > 0$, and define $||\theta||_{\mathbb{T}} = \min_{k\in \mathbb{Z}}|\theta - 2\pi k|$ for all angles $\theta\in \mathbb{R}/2\pi \mathbb{Z}$. Then, given the weighted adjacency matrix $A$ as in (\ref{weighred_adj_matr_noisy}), angular synchronization algorithm outputs $\tilde{\upsilon}\in \mathbb{C}^{|V|}$ with unit-modulus entries, such that for some phase ${\theta\in \mathbb{R}/2\pi \mathbb{Z}}$,
\begin{equation*}
\sum_{\lambda\in \Lambda''} ||\arg(\tilde{\upsilon}_\lambda) - \arg(\langle x, \pi(\lambda)g \rangle) - \theta||_{\mathbb{T}}^2\le \frac{C||\varepsilon||^2}{\tau^2 P^2},
\end{equation*}
\noindent where $P = \min_{(\lambda_1, \lambda_2)\in E''}|\overline{\langle x, \pi(\lambda_1)g \rangle} \langle x, \pi(\lambda_2)g \rangle + \varepsilon_{\lambda_1 \lambda_2}|$ and $C$ is a universal constant.
\end{theorem}

As Theorem \ref{angular_synchronization} shows, the accuracy of the angular synchronization algorithm depends on the spectral gap of the graph $G''$. To find a subgraph $G''\subset G'$ with spectral gap bounded away from zero, we shall use the \emph{spectral clustering} algorithm \cite{ng, mixon1}.
To ensure that $|\Lambda''|\ge M$, we rely on the following result. Its proof is based on the Cheeger inequality for the graph connection Laplacian \cite{bandeira} and can be found in \cite{mixon1}.
\begin{theorem}\label{graph_pruning}
Take $p\ge q\ge \frac{2}{3}$. Consider a regular graph $G = (V,E)$ with spectral gap $\lambda_2 > g(p,q) = 1 - 2(q(1-q) - (1 - p))$ and set $\tau = \frac{1}{8}(\lambda_2 - g(p,q))^2$. Then, after Algorithm \ref{deleting_small_and_large} removes at most $(1-p)|V|$ vertices from $G$, spectral clustering algorithm outputs a subgraph with at least $q|V|$ vertices and spectral gap at least $\tau$.
\end{theorem}

We summarize the above discussion in the following reconstruction algorithm.\\

\begin{algorithm}[H]
    \SetKwInOut{Input}{Input}
    \SetKwInOut{Output}{Output}

    \Input{$F\subset\mathbb{Z}_M$, $C\subset \mathbb{Z}_M$, window $g$; \\ noisy measurements $b$ w.r.t. $\Phi_\Lambda\cup\Phi_E$, given by~(\ref{measurements});\\ parameters $\alpha$, $\beta$,~$\tau$}
    \Output{approximation $\tilde{x}$ of the signal $x$ (up to a global phase)}
    construct graph $G = (\Lambda,E)$ with $\Lambda = F\times \mathbb{Z}_M$ and $E$ as in (\ref{edge_frame})\;
    assign to each $\lambda\in \Lambda$ weight $b_\lambda$\;
    assign to each edge $(\lambda_1, \lambda_2)\in E$ weight $A_{\lambda_1 \lambda_2}$, given in (\ref{weighred_adj_matr_noisy})\;
    run Algorithm \ref{deleting_small_and_large} with parameters $\alpha$, $\beta$ to obtain $G' = (\Lambda',E')$\;
    run \emph{spectral clustering} to find $G'' = (\Lambda'',E'')$ with $\Spg(G'')\ge \tau$\;
    run \emph{angular synchronization} to obtain approximate phases $\{u_\lambda\}_{\lambda\in\Lambda''}$\;
    set $c_\lambda = u_\lambda \sqrt{b_\lambda}, ~ \lambda \in \Lambda''$\;
    reconstruct $\tilde{x} = (\Phi_{\Lambda''} \Phi_{\Lambda''}^*)^{-1}\Phi_{\Lambda''}c$ from $c = \{c_\lambda\}_{\lambda\in \Lambda''}$.\\
    \caption{\label{noisy_reconstruction} Reconstruction in the noisy case}
\end{algorithm}

%
Now we are ready to prove our main result.

%

\begin{theorem} \label{stability_main}
Fix $x\in \mathbb{C}^M$ and consider the measurement procedure (\ref{measurements}) described above, with $|F|$ and $d$ sufficiently large. If the noise vector satisfies $\frac{||\nu||_2}{||x||_2^2}\le \frac{C_1}{M}$ for some $C_1$ small enough, then there exists a constant $C''$ such that the estimate $\tilde{x}$ produced by Algorithm \ref{noisy_reconstruction} from the noisy measurements $\{b_j\}_{j = 1}^N$ satisfies with overwhelming probability
\begin{equation*}
\min_{\theta\in [0,2\pi)}||\tilde{x} - e^{i\theta}x||_2^2 \le \frac{C''\sqrt{M}||\nu||_2}{\Delta},
\end{equation*}
\noindent where $\Delta = \min_{\Lambda''\subset \Lambda, |\Lambda''|\ge 2/3|\Lambda|}\sigma_{\min}^2(\Phi_{\Lambda''}^*)$.
\end{theorem}

\begin{proof}
Without loss of generality we can assume that $||x||_2 = 1$. As follows from Lemma \ref{bernoulli}, with overwhelming probability $\Spg(G)\ge 1 - \frac{6}{\sqrt{b}}$. Let us fix parameters $\tau_0 > 0$, $\alpha, \beta \in (0,1)$ and apply Theorem \ref{graph_pruning} with \linebreak$g(p,q) = 1 - 2(q(1-q) - (1 - p)) = 1 - \frac{6}{\sqrt{b}} - \tau_0 < \Spg(G)$. Then, after Algorithm~\ref{deleting_small_and_large} deletes $(1-p)|\Lambda| = (1 - \alpha - \beta)|\Lambda|$ vertices with the smallest and the largest corresponding measurements, we apply spectral clustering algorithm with parameter $\tau = \frac{1}{8}(\Spg(G) - g(p,q))^2\ge \frac{\tau_0^2}{8}$. We obtain a graph $G'' = (V'',E'')$ with $|V''|\ge q|\Lambda|$ and $\Spg(G'')\ge \frac{\tau_0^2}{8}$. \par

Let us specify $q$ now. Since we set $1 - 2(q(1-q) - (1 - p)) = 1 - \frac{6}{\sqrt{b}} - \tau_0$, it follows $q(1-q) = \frac{3}{\sqrt{b}} + (1 - \alpha - \beta)  + \frac{\tau_0}{2}$. If $\tau_0$, $\alpha$, $\beta$, and $b$ are chosen appropriately, we can make $\frac{3}{\sqrt{b}} + (1 - \alpha - \beta )+ \frac{\tau_0}{2} = A \le \frac{2}{9}$. Then \linebreak$1\ge q = \frac{1 + \sqrt{1 - 4A}}{2}\ge \frac{1 + \sqrt{1/9}}{2} = \frac{2}{3}$. This ensures that $q\in \left( \frac{2}{3},1\right)$.\par

Now, after applying angular synchronization algorithm and using Theorem \ref{angular_synchronization}, we obtain a universal constant $C>0$ and phase $\theta\in \mathbb{R}/2\pi \mathbb{Z}$, such that
\begin{equation*}
\sum_{\lambda\in V''} ||\arg(u_\lambda) - \arg(\langle x, \pi(\lambda)g \rangle) - \theta||_{\mathbb{T}}^2\le \frac{C||\varepsilon||_2^2}{\tau^2 P^2},
\end{equation*}
\noindent where $\varepsilon_{\lambda_1 \lambda_2} = \frac{1}{3}\sum_{t = 0}^2\omega^t\nu_{\lambda_1\lambda_2 t}$. Since $||\nu||_2\le \frac{C_1}{M}$, we also have $|\nu_{\lambda_1\lambda_2 t}|\le \frac{C_1}{M}$, for all $\lambda_1,\lambda_2 \in \Lambda''$ and $t\in \{0,1,2\}$, and thus $|\varepsilon_{\lambda_1 \lambda_2}|\le \frac{C_1}{M}$. By the Cauchy-Schwarz inequality we have
\begin{equation*}
||\varepsilon||_2^2 = \sum_{(\lambda_1,\lambda_2)\in E''}|\varepsilon_{\lambda_1\lambda_2}|^2\le \frac{1}{9}\sum_{(\lambda_1,\lambda_2)\in E''}\left(\sum_{i = 0}^{2}|\omega^t|^2\right)\left(\sum_{i = 0}^{2}|\nu_{\lambda_1 \lambda_2 t}|^2\right)\le\frac{1}{3}||\nu||_2^2.
\end{equation*}
Theorem \ref{number_of_small_meas} implies that, for any $k,c>0$ and $\epsilon = c^2$,  $$\left|\left\lbrace\lambda\in \Lambda \text{, s.t. } |\langle x, \pi(\lambda)g \rangle|< \frac{c}{\sqrt{M}}\right\rbrace\right|<|\Lambda|(\epsilon + k\sqrt{2\epsilon})$$ with probability at least $1-1/k^2$. Then, since $|\nu_{\lambda}|\le \frac{C_1}{M}$, on this event $$\left\lbrace\lambda\in \Lambda \text{, s.t. } |\langle x, \pi(\lambda)g \rangle|^2 + \nu_\lambda < \tfrac{c^2}{M} - \tfrac{C_1}{M}\right\rbrace \subset \left\lbrace\lambda\in \Lambda \text{, s.t. } |\langle x, \pi(\lambda)g \rangle|^2 < \tfrac{c^2}{M}\right\rbrace,$$ that is, $$\left|\left\lbrace\lambda\in \Lambda \text{, s.t. } |\langle x, \pi(\lambda)g \rangle|^2 + \nu_\lambda < \frac{c^2 - C_1}{M}\right\rbrace\right| <|\Lambda|(\epsilon + k\sqrt{2\epsilon}).$$
We set $\alpha = 1 - (\epsilon + k\sqrt{2\epsilon})$ and delete $|\Lambda|(\epsilon + k\sqrt{2\epsilon})$ vertices with the smallest corresponding measurements. For the remaining coefficients we have $|\langle x, \pi(\lambda)g \rangle|\ge \frac{\tilde{c}}{\sqrt{M}}$ for some constant $\tilde{c}>0$, provided $c^2$ is sufficiently larger then $C_1$. Similarly, setting $\beta = 1 - (\eta + k\sqrt{2\eta})$, we delete $|\Lambda|(\eta + k\sqrt{2\eta})$ vertices with the largest corresponding measurements, and for the remaining vertices, with probability at least $1 - \frac{1}{k^2}$, we have $|\langle x, \pi(\lambda)g \rangle|\le \frac{\tilde{K}}{\sqrt{M}}$, for some constant $\tilde{K}>0$.\par

Now, with probability at least $1-\frac{2}{k^2}$, after applying Algorithms 1 and 2, we have $\frac{\tilde{c}}{\sqrt{M}}\le |\langle x, \pi(\lambda)g \rangle|\le \frac{\tilde{K}}{\sqrt{M}}$ for all $\lambda\in V''$. Thus
\begin{equation*}
\begin{split}
P & = \min_{(\lambda_1,\lambda_2)\in E''}|\overline{\langle x, \pi(\lambda_1)g \rangle} \langle x, \pi(\lambda_2)g \rangle + \varepsilon_{\lambda_1\lambda_2}|\\
& \ge \left|\frac{\tilde{c}^2}{M}-\max_{(\lambda_1,\lambda_2)\in E''}|\varepsilon_{\lambda_1 \lambda_2}|\right| \ge \left|\frac{\tilde{c}^2}{M} - \frac{C_1}{M}\right|\ge \frac{\tilde{C}}{M}.
\end{split}
\end{equation*}
Summing up, we obtain 
\begin{equation*}
\sum_{\lambda\in V''} ||\arg(u_\lambda) - \arg(\langle x, \pi(\lambda)g \rangle) - \theta||_{\mathbb{T}}^2\le \frac{64C||\nu||_2^2M^2}{3\tau_0^4\tilde{C}^2}.
\end{equation*}

For every $\lambda\in V''$, we denote the obtained estimate of the corresponding frame coefficient by $c_\lambda = u_\lambda\sqrt{b_\lambda} = u_\lambda\sqrt{|\langle x,\pi(\lambda)g \rangle|^2 + \nu_\lambda}$. Also, set \linebreak$\delta_\lambda = c_\lambda - e^{i\theta}\langle x,\pi(\lambda)g \rangle$ and $\zeta_\lambda = \sqrt{b_\lambda} - |\langle x,\pi(\lambda)g \rangle|$. Then
\begin{align*}
|\delta_\lambda| & = \left|\sqrt{b_\lambda}e^{i \arg(u_\lambda)} - \sqrt{b_\lambda}e^{i (\theta + \arg(\langle x,\pi(\lambda)g \rangle))} + \zeta e^{i (\theta + \arg(\langle x,\pi(\lambda)g \rangle))} \right|\\
& \le \sqrt{b_\lambda}\left| e^{i (\arg(u_\lambda) - \arg(\langle x,\pi(\lambda)g \rangle) - \theta)} - 1\right| + |\zeta_\lambda|\\
& \le \sqrt{b_\lambda}||\arg(u_\lambda) - \arg(\langle x,\pi(\lambda)g \rangle) - \theta||_\mathbb{T} + |\zeta_\lambda|.
\end{align*}
Further, since $|\delta_\lambda|^2 \le 2b_\lambda||\arg(u_\lambda) - \arg(\langle x,\pi(\lambda)g \rangle) - \theta||_\mathbb{T}^2 + 2|\zeta_\lambda|^2$, it follows
\begin{equation*}
||\delta||_2^2 \le 2\sum_{\lambda\in V''}b_\lambda||\arg(u_\lambda) - \arg(\langle x,\pi(\lambda)g \rangle) - \theta||_\mathbb{T}^2 + 2\sum_{\lambda\in V''}|\zeta_\lambda|^2.
\end{equation*}
\noindent Using the fact that, for any $a,b\in \mathbb{R}$, $(a-b)^2\le |a^2 - b^2|$, we obtain \linebreak$\zeta_\lambda^2 = (\sqrt{b_\lambda} - |\langle x,\pi(\lambda)g \rangle|)^2\le~|\nu_\lambda|$. And, since $b_\lambda\le \frac{\tilde{K}^2}{M}$ and $||.||_1\le \sqrt{|\Lambda''|}||.||_2$,
\begin{equation*}
\begin{split}
||\delta||_2^2 & \le \frac{2\tilde{K}^2}{M}\sum_{\lambda\in V''}||\arg(u_\lambda) - \arg(\langle x,\pi(\lambda)g \rangle) - \theta||_\mathbb{T}^2 + 2||\nu_V||_1\\
& \le \frac{128C\tilde{K^2}||\nu||_2^2M}{3\tau_0^4\tilde{C}^2} + 2\sqrt{|F|M}||\nu||_2.
\end{split}
\end{equation*}
Since $||\nu||_2\le \frac{C_1}{M}$, we have $||\delta||_2^2 \le C''\sqrt{M}||\nu||_2$.

For the estimate $\tilde{x}$ of $x$, constructed by Algorithm \ref{noisy_reconstruction}, we have 
\begin{equation*}
\tilde{x} = (\Phi_{\Lambda''}\Phi_{\Lambda''}^*)^{-1}\Phi_{\Lambda''}c = e^{i\theta}x + (\Phi_{\Lambda''}\Phi_{\Lambda''}^*)^{-1}\Phi_{\Lambda''}\delta.
\end{equation*}
\noindent As such, we obtain the desired bound on the reconstruction error
\begin{equation*}
||\tilde{x} - e^{i\theta}x||_2^2 \le ||(\Phi_{\Lambda''}\Phi_{\Lambda''}^*)^{-1}\Phi_{\Lambda''}||_2^2||\delta||_2^2 = \frac{||\delta||_2^2}{\sigma_{\min}^2(\Phi_{\Lambda''}^*)} \le \frac{C''\sqrt{M}||\nu||_2}{\sigma_{\min}^2(\Phi_{\Lambda''}^*)}.
\end{equation*}
\end{proof}

\section{Numerical results on the robustness of Algorithm \ref{noisy_reconstruction}}\label{numerical_robustness}

In this section we numerically investigate the behavior of the constructed phase retrieval algorithm for time-frequency structured measurements (Algorithm~\ref{noisy_reconstruction}). In~particular, we use numerical simulations to demonstrate robustness of Algorithm~\ref{noisy_reconstruction} in the presence of noise and to investigate dependencies of the reconstruction error on various parameters.

Recall that, to construct the measurement frame $\Phi = \Phi_\Lambda \cup \Phi_E\subset \mathbb{C}^M$, we first introduce a random graph of measurements $G = (\Lambda, E)$, where ${\Lambda = F\times \mathbb{Z}_M}$ with $F\subset\mathbb{Z}_M$ such that $|F|$ is a constant that does not depend on the ambient dimension~$M$, and
\begin{equation*}
E = \left\lbrace\left((k_1,\ell_1),(k_2,\ell_2)\right) \text{, s.t. } k_1,k_2\in F, ~\ell_2 - \ell_1\in C\right\rbrace\subset \Lambda\times\Lambda.
\end{equation*}
Here, set $C$ is a random subset of $\mathbb{Z}_M$, such that
\begin{equation}\label{num_set_C}
C =  D\cup (-D)\setminus \{0\}\subset \mathbb{Z}_M \text{ with } {\bf 1}_D(m)\sim \text{i.i.d. } B\left(1,\tfrac{d\log M}{M}\right),
\end{equation}
\noindent that is, $D\subset\mathbb{Z}_M$ is constructed at random, so that every $m\in \mathbb{Z}_M$ is chosen to be an element of $D$ independently with probability $\tfrac{d\log M}{M}$, for some parameter~$d>0$.


The measurement frame $\Phi_\Lambda$ and the set of vectors for additional measurements $\Phi_E$ are then given by
\begin{equation*}
\begin{split}
& \Phi_\Lambda = (g,\Lambda),~ g \in \mathbb{C}^M \text{ uniformly distributed on } \mathbb{S}^{M-1}\subset \mathbb{C}^M;\\
& \Phi_E = \left\lbrace\pi(\lambda_1)g + \omega^t\pi(\lambda_2)g \right\rbrace_{(\lambda_1,\lambda_2)\in E,~ t\in \{0,1,2\}} \text{ with }\omega = e^{2\pi i/3}.
\end{split}
\end{equation*}
For our numerical simulations, we consider noisy phaseless measurements
\begin{equation*}
\begin{split}
b_\lambda = |\langle x, \pi(\lambda)g \rangle|^2 + \nu_\lambda, & \quad \lambda\in \Lambda;\\
b_{\lambda_1\lambda_2t} = |\langle x, \pi(\lambda_1)g + \omega^t \pi(\lambda_2)g \rangle|^2 + \nu_{\lambda_1\lambda_2t}, & \quad (\lambda_1,\lambda_2)\in E,~ t\in \{0,1,2\},
\end{split}
\end{equation*}
\noindent where $\nu_\lambda ,~ \nu_{\lambda_1\lambda_2t} \sim \text{i.i.d.} ~\mathcal{N}(0, \sigma)$ are independent normally distributed additive noise components. Theorem \ref{stability_main} then gives the following bound on the reconstruction error of Algorithm \ref{noisy_reconstruction}
\begin{equation}\label{theoretical_bound}
||\tilde{x} - e^{i\theta}x||_2^2 \le C\sqrt{M}||\nu||_2,
\end{equation}
\noindent where the constant $C$ depends on the spectral gap of the graph of measurements $G$ and on the parameter $$\Delta = \min_{\substack{\Lambda''\subset \Lambda,\\ |\Lambda''|\ge 2/3|\Lambda|}}\sigma_{\min}^2(\Phi_{\Lambda''}^*).$$


Singular values of the analysis matrices of Gabor frames with random windows are studied in \cite{salanevich_sing_val_gabor}. In particular, it has been shown that the smallest singular value of the analysis matrix $\sigma_{\min}^2(\Phi_{\Lambda''}^*)$ of a randomly selected subframe $(g, \Lambda'')$ of a Gabor frame $(g, \Lambda'')$ is bounded from below by $c\frac{|\Lambda''|}{M}$, for a suitably chosen numerical constant $c>0$ , with high probability. Unfortunately, no uniform bounds, similar to the parameter $\Delta$ defined above, are known to the date for Gabor frames. We formulate the following conjecture.

\begin{conjecture}\label{conj_sing_val_gabor_unif}
Consider a Gabor frame $(g, \Lambda)$ with $g$ uniformly distributed on~$\mathbb{S}^{M-1}$ and $\Lambda\subset \mathbb{Z}_M\times\mathbb{Z}_M$, such that $|\Lambda| = O(M\log^\alpha M)$ (where the parameter $\alpha\geq 0$ has to be specified). Then, for any $\varepsilon\in (0,1)$, $$\Delta(p)=\min_{\substack{\Lambda''\subset \Lambda,\\ |\Lambda''|\ge p|\Lambda|}}\sigma_{\min}^2(\Phi_{\Lambda''}^*)\geq c\frac{|\Lambda|}{M}$$ with probability at least $1-\varepsilon$, where $c>0$ is some numerical constant that depends only on $p$ and $\varepsilon$.
\end{conjecture}

\begin{figure}[t]\center
\begin{tabular}{cc}
\begin{overpic}
[width=67mm,tics=20
]{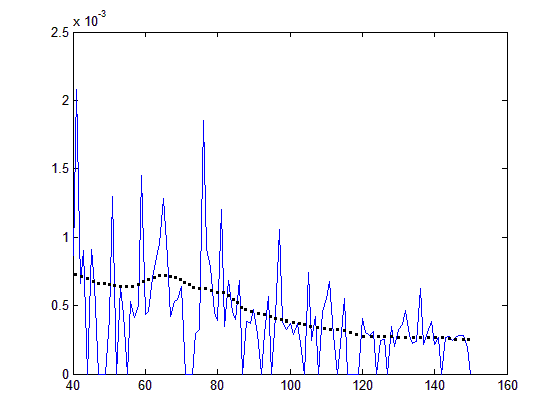}
    \put(10,75){\footnotesize Reconstruction error of Algorithm \ref{noisy_reconstruction} }
    \put(40,0){\scriptsize dimension $M$}
    \put(3,20){\rotatebox{90}{\scriptsize error $||\tilde{x} - e^{i\theta}x||_2^2$}}
  \end{overpic}
&
\begin{overpic}
[width=67mm,tics=20
]{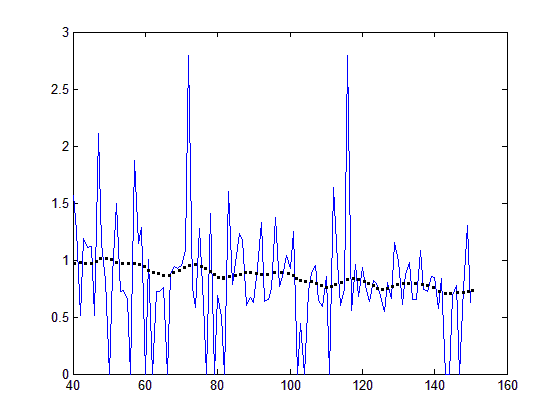}
    \put(10,75){\footnotesize Reconstruction error of Algorithm \ref{noisy_reconstruction} }
    \put(40,0){\scriptsize dimension $M$}
   \put(0,10){\rotatebox{90}{\scriptsize error to noise ratio $\frac{||\tilde{x} - e^{i\theta}x||_2^2}{||\nu||_2}$}}
  \end{overpic}
\end{tabular}
\caption{\label{num_dim}Dependence of the reconstruction error of Algorithm \ref{noisy_reconstruction} (left) and the error to noise ratio (right) on the ambient dimension $M$. Here the noise vector is random, such that it has independent normally distributed entries with variance $\sigma = 10^{-3}$. The black dashed lines on the plots show the average (over several simulations with different noise realizations) values of the reconstruction error and the error to noise ratio, respectively. These numerical results suggest that the error to noise ratio does not depend on the signal dimension and is bounded above by a numerical constant.}
\end{figure}
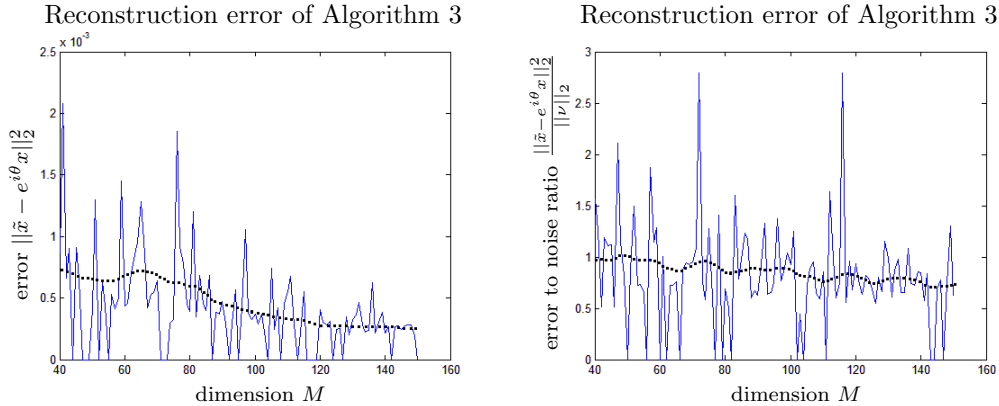

To illustrate Theorem \ref{stability_main}, we consider two sets of simulations. For the first one, we let the dimension of the signal vary and explore the reconstruction error of the algorithm for a random normally distributed noise vector with independent entries and fixed variance. On Figure \ref{num_dim} we show the obtained results, which suggest that the error to noise ratio does not depend on the signal dimension, unlike the bound \eqref{theoretical_bound} obtained in Theorem \ref{stability_main}. In fact, the ratio between the reconstruction error and the norm of the noise vector appears to be bounded above by a numerical constant close to $3$.

\begin{figure}[t]\center
\begin{tabular}{cc}
\begin{overpic}
[width=67mm,tics=20
]{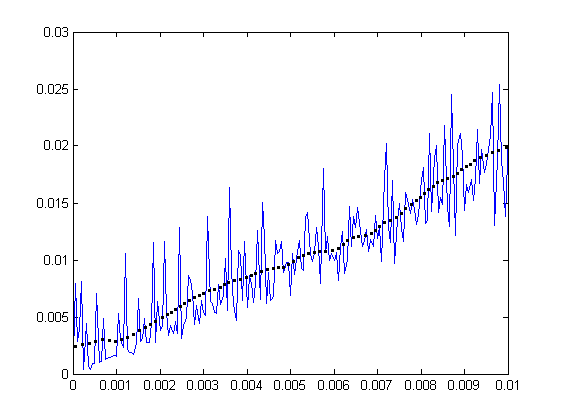}
    \put(10,75){\footnotesize Reconstruction error of Algorithm \ref{noisy_reconstruction} }
    \put(35,0){\scriptsize noise variance $\sigma$}
    \put(0,20){\rotatebox{90}{\scriptsize error $||\tilde{x} - e^{i\theta}x||_2^2$}}
  \end{overpic}
&
\begin{overpic}
[width=67mm,tics=20
]{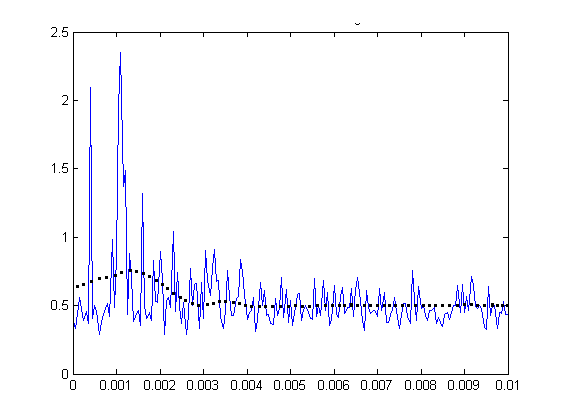}
    \put(10,75){\footnotesize Reconstruction error of Algorithm \ref{noisy_reconstruction} }
    \put(35,0){\scriptsize noise variance $\sigma$}
   \put(0,10){\rotatebox{90}{\scriptsize error to noise ratio $\frac{||\tilde{x} - e^{i\theta}x||_2^2}{||\nu||_2}$}}
  \end{overpic}
\end{tabular}
\caption{\label{num_var}Dependence of the reconstruction error of Algorithm \ref{noisy_reconstruction} (left) and the error to noise ratio (right) on the variance $\sigma$ of the entries $\nu_\lambda, \nu_{\lambda_1 \lambda_2 t}$ of the noise vector $\nu$, which are selected independently from the normal distribution $\mathcal{N}(0,\sigma)$. The ambient dimension here is $M=100$. The black dashed lines on the plots show the average (over several simulations with different noise realizations) values of the reconstruction error and the error to noise ratio, respectively. These numerical results suggest that the reconstruction error grows linearly with the magnitude of noise.}
\end{figure}
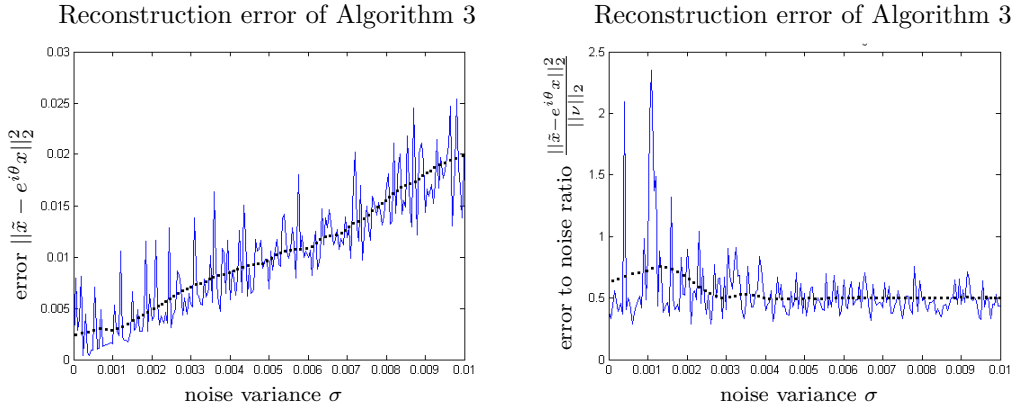

For the second set of simulation, we explore the dependence of the reconstruction error and the error to noise ratio on the noise variance for a fixed signal dimension. The obtained results, shown on Figure \ref{num_var}, illustrate that the reconstruction error grows linearly with the magnitude of noise. 

On both Figures \ref{num_dim} and \ref{num_var}, we show the average values of the reconstruction error and the error to noise ratio (over several simulations with different noise realizations) using black dashed lines. In other words, the black dashed lines on the plots show (an approximation of) the expected values of the corresponding quantities.  We note that, on both figures, the average error to noise ratio appears to be smaller than $1$, which means that noise reduction takes place during signal reconstruction. This can be explained in the following way. Assuming that the graph of measurements $G$ is sufficiently well connected, that is, $\Spg(G)$ is sufficiently big, the phase of a frame coefficient can be propagated to the corresponding vertex using various different paths. In Algorithm \ref{noisy_reconstruction}, we use the angular synchronization algorithm, which utilizes relative phase information coming to a vertex $\lambda\in \Lambda$ from all edges $(\lambda, \lambda')\in E$ incident to $\lambda$. Since in the simulations we considered noise with independent entries and zero mean, it tends to cancel itself at a vertex.

The reason why the plots on Figures \ref{num_dim} and \ref{num_var} look quite spiky is that different realizations of the random graph of measurements $G$ are used for the simulations. As we mentioned before, the reconstruction error bound~\eqref{theoretical_bound} of Algorithm \ref{noisy_reconstruction} depends on the spectral gap of $G$, which might differ from one realization to another. In particular, the bigger the cardinality of the random set $C$ is, the better are the connectivity properties of $G$, and the smaller is the reconstruction error.

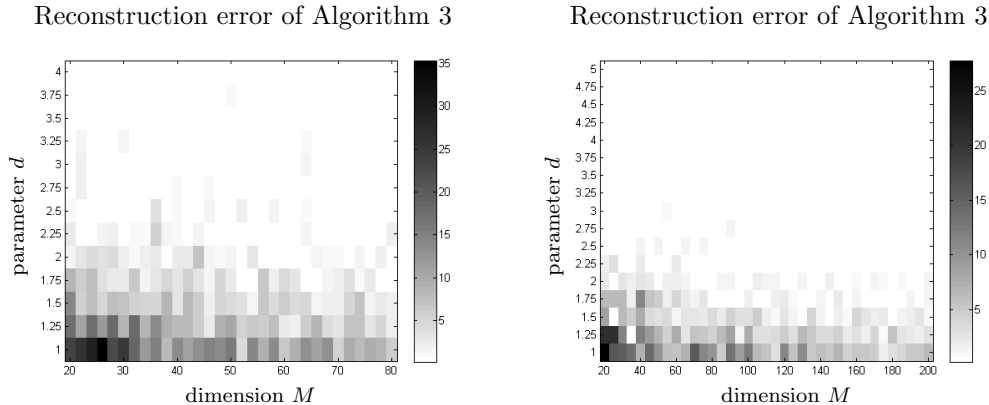
\begin{figure}[t]\center
\begin{tabular}{cc}
\begin{overpic}
[width=67mm,tics=20
]{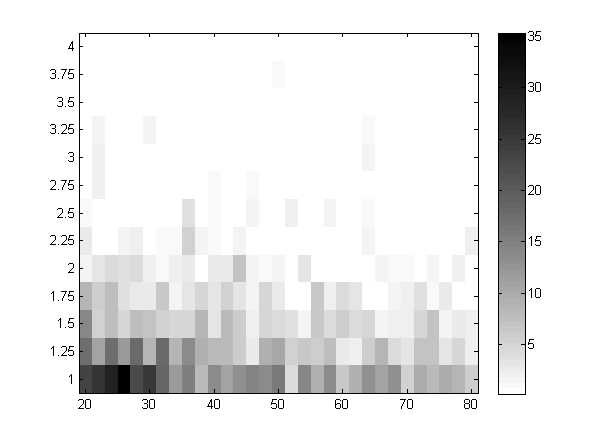}
    \put(7,75){\footnotesize Reconstruction error of Algorithm \ref{noisy_reconstruction} }
    \put(37,0){\scriptsize dimension $M$}
    \put(2,25){\rotatebox{90}{\scriptsize parameter $d$}}
  \end{overpic}
&
\begin{overpic}
[width=67mm,tics=20
]{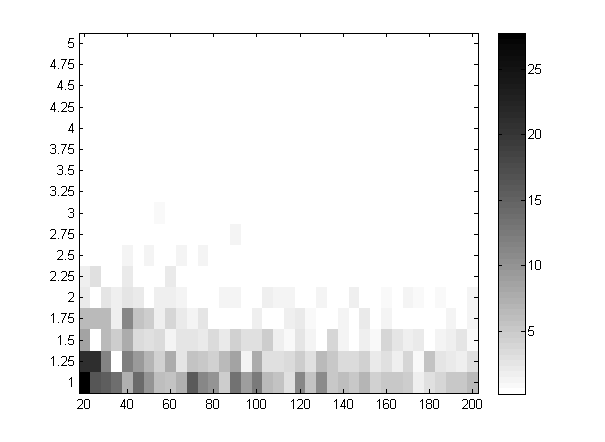}
    \put(7,75){\footnotesize Reconstruction error of Algorithm \ref{noisy_reconstruction} }
    \put(37,0){\scriptsize dimension $M$}
     \put(2,25){\rotatebox{90}{\scriptsize parameter $d$}}
\end{overpic}
\end{tabular}
\caption{\label{num_constant} Dependence of the reconstruction error to noise ratio of Algorithm~\ref{noisy_reconstruction} on the papameter $d$ in \eqref{num_set_C} for various dimensions. Here, values of high error to noise ratio are color coded darker then those with a low ratio. Parameter~$d$ controls the connectivity properties of the graph of measurements $G$. These numerical results suggest that, starting at approximately $d = 3$, the error to noise ratio does not exceed $4$, and also does not depend on the ambient dimension $M$. This observation allows to reduce the multiplicative constant in the number of measurements required for the reconstruction.}
\end{figure}

The cardinality of the set $C$ and, thus, the reconstruction error of the algorithm depend on the parameter $d$, as formula \eqref{num_set_C} shows.  More precisely, it follows from Lemma \ref{bernoulli} that, provided $d > 36$,
\begin{equation*}
\Spg (G) \ge 1 - \frac{6}{\sqrt{d}}
\end{equation*}
\noindent with overwhelming probability. We now investigate the dependence of the reconstruction error of the algorithm on the parameter $d$ numerically.

%

Numerical results presented on Figure \ref{num_constant} show the dependence of the ratio between the reconstruction error of Algorithm \ref{noisy_reconstruction} and the norm of the noise vector on the parameter $d$ (vertical axis) for the varying ambient dimension (horizontal axis). One can see that starting approximately at $d = 3$, that is, much earlier than the value $d = 144$, predicted by Theorem \ref{stability_main}, this ratio does not exceed $4$.

\section{Conclusions}

For most of the existing phase retrieval algorithms, including\linebreak PhaseLift~\cite{candes1, candes2} and Wirtinger flow algorithm \cite{candes2015phase}, recovery and robustness guarantees are proven for the case of a random measurement frame with independent frame vectors. One of the main reasons for this is that properties of such frames are sufficiently well studied. Moreover, such frame appear to have optimal properties in the sense of being ``well-spread'', which is formalized in different ways for different reconstruction algorithms.

In this paper, we use the idea of polarization \cite{mixon1} to design the first phase retrieval algorithm for time-frequency structured frames $\Phi$ with $|\Phi|< M^2$. An investigation of properties of Gabor frames, that is, frames consisting of vectors which are time and frequency shifts of the random window, and, thus are not independent, allow us to conclude recovery and robustness guarantees for the postulated phase retrieval algorithm. In particular, Theorem \ref{number_of_small_meas}, which describes the order statistics of frame coefficients for a Gabor frame with a random window, allows us to obtain robustness guarantees in Theorem~\ref{stability_main}.

The numerical results presented in Section~\ref{numerical_robustness} suggest that the theoretical bound~\eqref{theoretical_bound} on the reconstruction error of the proposed phase retrieval algorithm can be further improved by a factor of $\sqrt{M}$. Thus, one of the important tasks for future research is to understand the gap between theoretically predicted robustness guarantees and results obtained numerically. We hope that a further study of properties of Gabor frames with random windows will allow us to not only remove the factor of $\sqrt{M}$ from the reconstruction error bound \eqref{theoretical_bound}, but also to prove the following conjecture, which states the uniform robustness guarantees for Algorithm~\ref{noisy_reconstruction}.

\begin{conjecture}\label{conjecture_unif_robustness}
Consider the measurement procedure \eqref{measurements} with $|F|$ and $d$ sufficiently large. If the noise vector satisfies $\frac{||\nu||_2}{||x||_2^2}\le \frac{C_1}{M}$ for some $C_1$ small enough, then there exists a numerical constant $C>0$ such that with overwhelming probability, for every $x\in \mathbb{C}^M$, the estimate $\tilde{x}$ produced by Algorithm~\ref{noisy_reconstruction} satisfies
\begin{equation*}
\min_{\theta\in [0,2\pi)}||\tilde{x} - e^{i\theta}x||_2^2 \le C ||\nu||_2.
\end{equation*}
\end{conjecture} 

We note that one of the main ingredients of the proof of Theorem~\ref{stability_main} is Theorem~\ref{number_of_small_meas}, which gives bounds on the frame order statistics of a Gabor frame with a random window (see Section~\ref{projective_uniformity}). Similarly, the main missing ingredient of the proof of Conjecture~\ref{conjecture_unif_robustness} is a uniform version of Theorem~\ref{number_of_small_meas}.

%

\section*{Acknowledgments}

The authors thank the anonymous referees for providing thoughtful suggestions that led to a more complete discussion of the context of our results. P. Salanevich thanks Prof. Dr. Felix Krahmer for insightful discussions during the ``Mathematics of Signal Processing'' trimester program (Hausdorff Research Institute for Mathematics, Bonn, Germany), and Prof. Dr. Terence Tao for his valuable suggestions during her visit at UCLA in Spring 2015. P. Salanevich also thanks Prof. Dr. Holger Boche for helpful discussions during her visit to Technical University of Munich in December 2014. P. Salanevich is supported by Research Grant for Doctoral Candidates and Young Academics and Scientists of German Academic Exchange Service (DAAD), the funding is greatly appreciated.

\appendix
\section{Proof of Theorem \ref{number_of_small_meas}}
\noindent Here we prove Theorem \ref{number_of_small_meas} which is formulated as follows.
\begin{theorem*}
Fix $x\in \mathbb{S}^{M-1}\subset \mathbb{C}^M$ and consider a Gabor frame $(g, \Lambda)$ with $\Lambda\subset\mathbb{Z}_M\times \mathbb{Z}_M$ and a random window $g$ uniformly distributed on the unit sphere $\mathbb{S}^{M-1}$. Then the following holds.
\begin{enumerate}
\item[(a)] For any $c>0$ and $k>0$, with probability at least $1-\frac{1}{k^2}$, we have
\begin{equation*}
\left|\left\lbrace\lambda\in \Lambda \text{, s.t. } |\langle x, \pi(\lambda)g \rangle|< \frac{c}{\sqrt{M}}\right\rbrace\right|<|\Lambda|(c^2 + kc).
\end{equation*}
\item[(b)] For any $K>0$ and $k>0$, with probability at least $1-\frac{1}{k^2}$, we have
\begin{equation*}
\left|\left\lbrace\lambda\in \Lambda \text{, s.t. } |\langle x, \pi(\lambda)g \rangle|> \frac{K}{\sqrt{M}}\right\rbrace\right|<|\Lambda|\left(\frac{8}{\pi}e^{-K^2} + k\frac{2\sqrt{2}}{\sqrt{\pi}}e^{-\frac{K^2}{2}}\right).
\end{equation*}
\end{enumerate}
\end{theorem*}

\begin{proof}
(a) For $x\in \mathbb{S}^{M-1}$ fixed, we set $$G_\delta(x) = \{\varphi\in \mathbb{S}^{M-1}\subset \mathbb{C}^M, \text{ s.t. } |\langle x, \varphi \rangle|<\delta\}.$$ We are interested in the distribution of the random variable
\begin{equation*}
Z_x = |G_\delta(x)\cap (g, \Lambda)| = \sum_{\lambda\in \Lambda}{\bf 1}_{\{\pi(\lambda)g\in G_\delta (x)\}},
\end{equation*}
\noindent where ${\bf 1}_{\{\varphi\in G_\delta (x)\}}$ is the characteristic function of the event ${\{\varphi\in G_\delta (x)\}}$. In words, $Z_x$ is the number of small measurements of the fixed signal $x$ with respect to the Gabor frame $(g,\Lambda)$ with a random window $g$.

First, note that for each $\lambda\in \Lambda$, $\pi(\lambda)g$ is also uniformly distributed on $\mathbb{S}^{M-1}$, that is, the random vectors $\pi (\lambda)g$, $\lambda\in \Lambda$, are equally distributed. Indeed, consider a random vector $h$, such that $h(m)\sim i.i.d. ~ \mathbb{C}\mathcal{N}(0,\frac{1}{M})$. Then it is well known that $h/||h||_2$ is uniformly distributed on $\mathbb{S}^{M-1}$, since random vector $h/||h||_2$ almost surely has unit norm and its distribution is rotation invariant \cite{marsaglia1972choosing}. In other words, we can write $g = h/||h||_2$. As such, we have $$\pi(\lambda)g = \pi (\lambda)h/||h||_2 = \pi (\lambda)h/||\pi (\lambda)h||_2,$$ since both modulation and time shift are unitary operators. The coordinates of $h$ are independent identically distributed random variables, thus translation $T_k$, which is just a permutation of the vector coordinates, preserves the distribution of $h$. For the modulation we have $M_\ell h(m) = e^{2\pi i m\ell/M}h(m)$ is also normally distributed, with \begin{equation*}
\begin{split}
\mathbb{E}(e^{2\pi i m\ell/M}h(m))  & = e^{2\pi i m\ell/M}\mathbb{E}( h(m)) = 0 ;\\
\Var (e^{2\pi i m\ell/M}h(m)) & = \mathbb{E}(e^{2\pi i m\ell/M}h(m)e^{-2\pi i m\ell/M}\bar{h}(m)) = \Var(h(m)) = \frac{1}{M}.
\end{split}
\end{equation*}
\noindent Thus the distribution of $h$ is preserved by both modulation and time shift, and $\pi(\lambda)g  = \pi (\lambda)h/||\pi (\lambda)h||$ has the same distribution as $g = h/||h||_2$.\par
Since $\pi(\lambda)g$ has the same distribution as $g$, we have $$\mathbb{P}\{|\langle x, \pi(\lambda)g \rangle|<\delta\} = \mathbb{P}\{|\langle x, g \rangle|<\delta\},$$ for all $\lambda\in \Lambda$. Thus
\begin{equation}\label{expect}
\mathbb{E}(Z_x) = \sum_{\lambda\in \Lambda}\mathbb{P}\{\pi(\lambda)g\in G_\delta (x)\} = |\Lambda|\mathbb{P}\{|\langle x,g \rangle|<\delta\}.
\end{equation}
Let $R$ be an orthogonal matrix, such that $x=Re_1$, where $e_1 = (1,0,0,\dots,0)^T$ is the first vector of the standard basis. Then $$|\langle x,g \rangle| = |\langle Re_1,g \rangle| = |\langle e_1,R^*g \rangle|.$$ By the rotational symmetry of the distribution of $g$, we obtain 
\begin{equation*}
\mathbb{P}\{|\langle x,g \rangle|<\delta\} = \mathbb{P}\{|\langle e_1,g \rangle|<\delta\} = \mathbb{P}\{|g(0)|<\delta\}.
\end{equation*}

Let us identify the complex unit sphere $\mathbb{S}^{M-1}\subset \mathbb{C}^M$ with the real unit sphere $\mathbb{S}_{\mathbb{R}}^{2M-1}\subset \mathbb{R}^{2M}$, using the map $\mathcal{I}:\mathbb{S}^{M-1}\to \mathbb{S}_{\mathbb{R}}^{2M-1}$ given by
\begin{equation}\label{complex_to_real}
\mathcal{I}(z_0,\dots,z_{M-1}) = (\Re(z_0),\Im(z_0),\dots,\Re(z_{M-1}),\Im(z_{M-1})).
\end{equation}
Since $g$ is uniformly distributed on $\mathbb{S}^{M-1}$, $\tilde{g} = \mathcal{I}(g)$ is uniformly distributed on $\mathbb{S}_{\mathbb{R}}^{2M-1}$. Thus
\begin{equation*}
\mathbb{P}\{|g(0)|<\delta\}  = \mathbb{P}\{\tilde{g}(0)^2 + \tilde{g}(1)^2<\delta^2\} = \frac{S_{<\delta}}{S_1},
\end{equation*}
\noindent where $S_1 = \frac{2\pi^M}{(M-1)!}$ is the surface area of $\mathbb{S}_{\mathbb{R}}^{2M-1}$, and $S_{<\delta}$ is the surface area of the set $\{z\in \mathbb{S}_{\mathbb{R}}^{2M-1}, \text{ s.t. } z_0^2 + z_1^2 <~\delta^2\}$.
\begin{equation*}
S_{<\delta} = \int_{-\delta}^{\delta} \int_{-\sqrt{\delta^2 - z_0^2}}^{\sqrt{\delta^2 - z_0^2}} \frac{2\pi^{M-1}}{(M-2)!}(1-z_0^2 - z_1^2)^{\frac{2M-1}{2}} dz_1 dz_0 < \frac{2\pi^M\delta^2}{(M-2)!},
\end{equation*}
\noindent that is, $\mathbb{P}\{|g(0)|<\delta\}\le \delta^2(M-1)$.
\longhide{
\noindent Since $g = h/||h||_2$ with $h(m)\sim i.i.d. ~ \mathbb{C}\mathcal{N}(0,\frac{1}{M})$, we have $g(0) = h(0)/||h||_2 = (a+ib)/||h||_2$, \linebreak where $a,b\sim i.i.d ~ \mathcal{N}(0,\frac{1}{2M})$. Then 
\begin{equation*}
\begin{split}
 \mathbb{P}\{|g(0)|<\delta\} &= \mathbb{P}\left\lbrace \frac{a^2 + b^2}{||h||_2^2}<\delta^2\right\rbrace \le\mathbb{P}\{a^2 + b^2< 4 \delta^2 \text{ or } ||h||_2^2 > 4\}\le\\
& \le \mathbb{P}\{a^2 + b^2< 4 \delta^2\} + \mathbb{P}\{||h||_2^2 > 4\}.
\end{split}
\end{equation*}
Since $a$ and $b$ are independent and have the same distribution, we have $\mathbb{P}\{a^2 + b^2< 4 \delta^2\}\le \mathbb{P}\{\min\{a^2, b^2\} < 2 \delta^2\}\le 2\mathbb{P}\{|a|<\sqrt{2}\delta\} = 2 \int_{-\sqrt{2}\delta}^{\sqrt{2}\delta} \frac{\sqrt{M}}{\sqrt{\pi}}e^{-x^2 M}dx\le 4\sqrt{2}\delta\frac{\sqrt{M}}{\sqrt{\pi}} = \delta\frac{4\sqrt{2M}}{\sqrt{\pi}}$.\par 

To bound the other term, we are going to use the following lemma from \cite{laurent2000adaptive}.

\begin{lemma}\label{chi_square}
Let $Y_1,\dots,Y_M$ be i.i.d. $\mathcal{N}(0,1)$ and let $c = (c_1,\dots,c_M)$ with $c_k$ non-negative, and
\begin{equation*}
Z = \sum_{k = 1}^M c_k(Y_k^2 - 1).
\end{equation*}
\noindent Then the following inequalities hold for any $t>0$
\begin{equation}\label{lower}
\mathbb{P}\{Z \ge 2||c||_2 \sqrt{t} + 2||c||_\infty t\} \le e^{-t};
\end{equation}
\begin{equation}\label{upper}
\mathbb{P}\{Z \le -2||c||_2 \sqrt{t}\} \le e^{-t}.
\end{equation}
\end{lemma}
Note, that $2M ||h||_2^2 = 2M \sum_{k = 1}^M (|a_k|^2 + |b_k|^2)$, where $h_k = a_k + ib_k$ and $a_k, b_k\sim i.i.d. ~ \mathcal{N}(0,\frac{1}{2M})$. Then for $k \in \{1,\dots,M\}$, $2Ma_k, 2Mb_k$ are independent standard Gaussian random variables and we can apply inequality (\ref{lower}) from Lemma \ref{chi_square} with $c_k = 1$ to get for any $t>0$
\begin{equation*}
\mathbb{P}\{2M||h||_2^2 \ge \sqrt{8Mt} + 2t + 2M\}\le e^{-t}.
\end{equation*}
\noindent Taking $t= M/2$ we obtain
\begin{equation*}
\mathbb{P}\{||h||_2^2 > 4\} = \mathbb{P}\{2M||h||_2^2 > 8M\} \le \mathbb{P}\{2M||h||_2^2 \ge 5M\}\le e^{-M/2}.
\end{equation*}
Then, summarizing, we obtain
\begin{equation}\label{prob_bound}
\mathbb{P}\{|\langle x,g \rangle|<\delta\} = \mathbb{P}\{|g(0)|<\delta\}\le \mathbb{P}\{a^2 + b^2< 4 \delta^2\} + \mathbb{P}\{||h||_2^2 > 4\}\le \delta\frac{4\sqrt{M}}{\sqrt{\pi}} + e^{-M/2}.
\end{equation}
}
Now, setting $\delta = \frac{c}{\sqrt{M}}$ for $c$ sufficiently small, we obtain
\begin{equation}\label{epsilon_prob_bound}
\mathbb{P}\left\lbrace|\langle x,g \rangle|<\frac{c}{\sqrt{M}}\right\rbrace\le c^2.
\end{equation}
Using equations (\ref{expect}) and (\ref{epsilon_prob_bound}), we obtain
\begin{equation}\label{expectation}
\mu = \mathbb{E}(Z_x) = |\Lambda|\mathbb{P}\left\lbrace|\langle x,g \rangle|<\frac{c}{\sqrt{M}}\right\rbrace\le |\Lambda|c^2.
\end{equation}
Similarly, using (\ref{epsilon_prob_bound}), for the variance of $Z_x$ we obtain 
\begin{align*}
\sigma^2 & = \Var(Z_x) = \mathbb{E}(Z_x^2) - (\mathbb{E}(Z_x))^2 \le \mathbb{E}(Z_x^2) = \mathbb{E}\left(\left(\sum_{\lambda \in \Lambda}{\bf 1}_{\{\pi(\lambda)g\in G_{\frac{c}{\sqrt{M}}}(x)\}}\right)^2\right) \\
& = \mathbb{E}\left(\sum_{\lambda \in \Lambda}{\bf 1}_{\left\lbrace\pi(\lambda)g\in G_{\frac{c}{\sqrt{M}}}(x)\right\rbrace}^2 + \sum_{\substack{(\lambda_1,\lambda_2)\in \Lambda^2,\\\lambda_1\ne\lambda_2}}{\bf 1}_{\left\lbrace\pi(\lambda_1)g\in G_{\frac{c}{\sqrt{M}}}(x)\right\rbrace} {\bf 1}_{\left\lbrace\pi(\lambda_2)g\in G_{\frac{c}{\sqrt{M}}}(x)\right\rbrace}\right)\\
& = \sum_{\lambda\in \Lambda}\mathbb{P}\left\lbrace\pi(\lambda)g\in G_{\frac{c}{\sqrt{M}}} (x)\right\rbrace + \sum_{\substack{(\lambda_1,\lambda_2)\in \Lambda^2,\\\lambda_1\ne\lambda_2}}\mathbb{P}\left\lbrace\pi(\lambda_1)g, \pi(\lambda_2)g \in G_{\frac{c}{\sqrt{M}}} (x) \right\rbrace\\
& \le (|\Lambda + (|\Lambda|^2 - |\Lambda|))|\mathbb{P}\left\lbrace|\langle x,g \rangle|<\frac{c}{\sqrt{M}}\right\rbrace \le c^2 |\Lambda|^2,\numberthis \label{variance}
\end{align*}
\noindent that is, $\sigma\le c|\Lambda|$. Then, using Chebychev inequality and bounds (\ref{expectation}) and (\ref{variance}), we have
\begin{equation*}\label{chebychev}
\mathbb{P}\{Z_x \ge |\Lambda|(c^2 + kc)\}\le \mathbb{P}\{Z_x \ge \mu + k\sigma\}\le \mathbb{P}\{|Z_x - \mu|\ge k\sigma\}\le \frac{1}{k^2},
\end{equation*}
\noindent for any $k>0$. In other words, if we delete $|\Lambda|(c^2 + kc)$ smallest phaseless measurements, for the remaining  measurements with probability at least \linebreak$1-\frac{1}{k^2}$ we would have $|\langle x, \pi (\lambda)g\rangle|\ge\frac{c}{\sqrt{M}}$. This concludes the proof of (a).\par

\vspace{0.5cm}
The proof of (b) follows the same steps. Let $K$ be a constant and consider the following random variable
\begin{equation*}
U_x = \sum_{\lambda\in \Lambda}{\bf 1}_{\{|\langle x, \pi(\lambda)g \rangle|>K/\sqrt{M}\}}.
\end{equation*}
Since, for each $\lambda\in \Lambda$, $\pi(\lambda)g$ has the same (uniform on $\mathbb{S}^{M-1}$) distribution as $g$, we have $\mathbb{P}\{|\langle x, \pi(\lambda)g \rangle|>K/\sqrt{M}\} = \mathbb{P}\{|\langle x, g \rangle|>K/\sqrt{M}\}$, for all $\lambda\in \Lambda$. As above, we have
\begin{equation*}
\mathbb{P}\{|\langle x,g \rangle|>K/\sqrt{M}\} = \mathbb{P}\{|\langle e_1,g \rangle|>K/\sqrt{M}\} = \mathbb{P}\{|g(0)|>K/\sqrt{M}\}.
\end{equation*}
Using the map $\mathcal{I}:\mathbb{S}^{M-1}\to\mathbb{S}_{\mathbb{R}}^{2M-1}$ defined in (\ref{complex_to_real}), for $\tilde{g} = \mathcal{I}(g)$ we obtain
\begin{equation*}
\mathbb{P}\left\lbrace |g(0)|>\frac{K}{\sqrt{M}}\right\rbrace  = \mathbb{P}\left\lbrace\tilde{g}(0)^2 + \tilde{g}(1)^2>\frac{K^2}{M}\right\rbrace = \frac{S_{>K/\sqrt{M}}}{S_1},
\end{equation*}
\noindent where $S_1 = \frac{2\pi^M}{(M-1)!}$ is the surface area of $\mathbb{S}_{\mathbb{R}}^{2M-1}$, and $S_{>K/\sqrt{M}}$ is the surface area of the set $\left\lbrace z\in \mathbb{S}_{\mathbb{R}}^{2M-1}, \text{ s.t. } z_0^2 + z_1^2 >~\frac{K^2}{M}\right\rbrace$.
\begin{equation*}
\begin{split}
S_{>K/\sqrt{M}} & = \int_{|z_0| \le1} \int_{\sqrt{\frac{K^2}{M} - z_0^2} < |z_1|<\sqrt{1 - z_0^2}} \frac{2\pi^{M-1}}{(M-2)!}(1-z_0^2 - z_1^2)^{\frac{2M-1}{2}} dz_1 dz_0 \\
& = 8\frac{2\pi^{M-1}}{(M-2)!} \int_0^{\frac{K}{\sqrt{2M}}}  \int_{\sqrt{\frac{K^2}{M} - z_0^2}}^{\sqrt{1 - z_0^2}} (1 - z_0^2 - z_1^2)^{\frac{2M-1}{2}} dz_1dz_0 \\
& +  8\frac{2\pi^{M-1}}{(M-2)!} \int_{\frac{K}{\sqrt{2M}}}^{\frac{1}{\sqrt{2}}}  \int_{z_0}^{\sqrt{1 - z_0^2}} (1 - z_0^2 - z_1^2)^{\frac{2M-1}{2}} dz_1dz_0 \\
& \le \frac{16\pi^{M-1}}{(M-2)!} \frac{\sqrt{2M}}{K}  \int_0^{\frac{K}{\sqrt{2M}}}  \int_{\sqrt{\frac{K^2}{M} - z_0^2}}^{\sqrt{1 - z_0^2}} z_1(1 - z_0^2 - z_1^2)^{\frac{2M-1}{2}} dz_1dz_0 \\
& + \frac{16\pi^{M-1}}{(M-2)!} \frac{\sqrt{2M}}{K} \int_{\frac{K}{\sqrt{2M}}}^{\frac{1}{\sqrt{2}}}  \int_{z_0}^{\sqrt{1 - z_0^2}}z_1 (1 - z_0^2 - z_1^2)^{\frac{2M-1}{2}} dz_1dz_0 \\
& \le \frac{8\pi^{M-1}}{(M-1)!}\left( \left(1 - \frac{K^2}{M}\right)^{\frac{2M+1}{2}} +  \frac{M}{2K^2}\int_{\frac{K}{\sqrt{2M}}}^{\frac{1}{\sqrt{2}}} 4z_0\left(1 - 2z_0^2\right)^{\frac{2M+1}{2}}dz_0\right) \\
& \le \frac{8\pi^{M-1}}{(M-1)!}\left( \left(1 - \frac{K^2}{M}\right)^{\frac{2M+1}{2}} +  \frac{1}{2K^2}\left(1 - \frac{K^2}{M}\right)^{\frac{2M+3}{2}}\right)  \\
& \le \frac{8\pi^{M-1}}{(M-1)!}\left( e^{-\frac{K^2}{M}\frac{2M+1}{2}} +  \frac{1}{2K^2}e^{-\frac{K^2}{M}\frac{2M+3}{2}}\right)\le \frac{16\pi^{M-1}}{(M-1)!} e^{-K^2}.
\end{split}
\end{equation*}
\noindent Here, we used the symmetry of the domain of integration, the fact that $z_0^2 + z_1^2 < \frac{K^2}{M}$ implies $\max \{|z_0|,|z_1|\}> \frac{K}{\sqrt{2M}}$, and inequality $1 - x \le e^{-x}$. Using the computed bound for $S_{>K/\sqrt{M}}$, we obtain $\mathbb{P}\left\lbrace |g(0)|>\frac{K}{\sqrt{M}}\right\rbrace \le \frac{8}{\pi}e^{-K^2}$. Then
\longhide{
\noindent Since $g = h/||h||_2$ with $h(m)\sim i.i.d. ~ \mathbb{C}\mathcal{N}(0,\frac{1}{M})$, we have $g(0) = h(0)/||h||_2 = (a+ib)/||h||_2$, \linebreak where $a,b\sim i.i.d ~ \mathcal{N}(0,\frac{1}{2N})$. Then 
\begin{equation*}
 \mathbb{P}\{|g(0)|>K/\sqrt{M}\} = \mathbb{P}\left\lbrace \frac{a^2 + b^2}{||h||_2^2}>\frac{K^2}{M}\right\rbrace \le \mathbb{P}\left\lbrace a^2 + b^2 > \frac{K^2}{4M}\right\rbrace + \mathbb{P}\{||h||_2^2 < 1/4\}.
\end{equation*}
Since $a$ and $b$ are independent and have the same distribution, we have 
\begin{equation*}
\begin{split}
{P}\left\lbrace a^2 + b^2>\frac{K^2}{4M}\right\rbrace \le \mathbb{P}\left\lbrace \max(a^2, b^2)>\frac{K^2}{8M}\right\rbrace\le 2\mathbb{P}\left\lbrace |a|>\frac{K}{\sqrt{8M}}\right\rbrace = \frac{4\sqrt{M}}{\sqrt{\pi}} \int_{\frac{K}{\sqrt{8M}}}^{\infty} e^{-x^2 M}dx \le\\
\le \frac{4\sqrt{M}}{\sqrt{\pi}} \int_{\frac{K}{\sqrt{8M}}}^{\infty} \frac{\sqrt{8M}x}{K}e^{-x^2 M}dx = -\frac{4\sqrt{2}}{K\sqrt{\pi}} \int_{\frac{K}{\sqrt{8M}}}^{\infty} e^{-x^2 M}d(-x^2 M) = \frac{4\sqrt{2}}{K\sqrt{\pi}}e^{-\frac{K^2}{8}}.
\end{split}
\end{equation*}

Since $2M ||h||_2^2 = 2M \sum_{k = 1}^M (|a_k|^2 + |b_k|^2)$, where $h_k = a_k + ib_k$ and $a_k, b_k\sim i.i.d ~ \mathcal{N}(0,1/2M)$. Then for $k \in \{1,\dots,M\}$, $2Ma_k, 2Mb_k$ are independent standard Gaussian random variables and we can apply inequality (\ref{upper}) from Lemma \ref{chi_square} with $c_k = 1$ to get
\begin{equation*}
\mathbb{P}\left\lbrace||h||_2^2 \le -\sqrt{\frac{2t}{M}} + 1\right\rbrace\le e^{-t}.
\end{equation*}
\noindent  for every $t>0$. Taking $t= 9M/32$, we obtain
\begin{equation*}
\mathbb{P}\left\lbrace||h||_2^2 \le -\sqrt{\frac{2t}{M}} + 1\right\rbrace = \mathbb{P}\left\lbrace||h||_2^2 \le \frac{1}{4}\right\rbrace\le e^{-9M/32},
\end{equation*}
\noindent and
\begin{equation}\label{prob_bound_large}
\begin{split}
\mathbb{P}\{|\langle x,g \rangle|>K/\sqrt{M}\} & = \mathbb{P}\{|g(0)|>K/\sqrt{M}\}\le \mathbb{P}\{a^2 + b^2>K^2/(4M)\} + \mathbb{P}\{||h||_2^2 < 1/4\}\le\\
& \le \frac{4\sqrt{2}}{K\sqrt{\pi}}e^{-\frac{K^2}{8}} + e^{-9M/32}.
\end{split}
\end{equation}
For simplicity of notation, let us set $\eta = \frac{4\sqrt{2}}{K\sqrt{\pi}}e^{-\frac{K^2}{8}} + e^{-9M/32}$. Then
}
\begin{equation}\label{expectation_large}
\mu = \mathbb{E}(U_x) = \sum_{\lambda\in \Lambda}\mathbb{P}\left\lbrace|\langle x,\pi(\lambda)g\rangle|>\tfrac{K}{\sqrt{M}}\right\rbrace \le \frac{8}{\pi}e^{-K^2}|\Lambda|.
\end{equation}
Similarly, for the variance of $U_x$ we obtain 
\begin{align*}
& \Scale[0.9]{\sigma^2 = \Var(U_x) \le \mathbb{E}(U_x^2) = \mathbb{E}\left(\left(\sum_{\lambda \in \Lambda}{\bf 1}_{\{|\langle x,\pi(\lambda)g\rangle|>K/\sqrt{M}\}}\right)^2\right)} \\
& \Scale[0.9]{= \sum_{\lambda\in \Lambda}\mathbb{P}\left\lbrace|\langle x,\pi(\lambda)g\rangle|>\tfrac{K}{\sqrt{M}}\right\rbrace + \sum_{\substack{(\lambda_1,\lambda_2)\in  \Lambda^2,\\\lambda_1\ne\lambda_2}}\mathbb{P}\left\lbrace|\langle x,\pi(\lambda_1)g\rangle|,  |\langle x,\pi(\lambda_2)g\rangle| >\tfrac{K}{\sqrt{M}}\right\rbrace}\\
& \Scale[0.9]{\le |\Lambda|\mathbb{P}\{|\langle x,g \rangle|>K/\sqrt{M}\} + (|\Lambda|^2 - |\Lambda|)\mathbb{P}\{|\langle x,g \rangle|>K/\sqrt{M}\}\le \frac{8}{\pi}e^{-K^2} |\Lambda|^2,}\numberthis \label{variance_large}
\end{align*}
\noindent that is, $\sigma\le \frac{2\sqrt{2}}{\sqrt{\pi}}e^{-\frac{K^2}{2}}|\Lambda|$. Then, using Chebychev inequality and bounds (\ref{expectation_large}) and (\ref{variance_large}), we obtain
\begin{equation*}\label{chebychev_large}
\mathbb{P}\left\lbrace U_x \ge |\Lambda|\left(\frac{8}{\pi}e^{-K^2} + k\frac{2\sqrt{2}}{\sqrt{\pi}}e^{-\frac{K^2}{2}}\right)\right\rbrace\le \mathbb{P}\{|U_x - \mu|\ge k\sigma\}\le \frac{1}{k^2},
\end{equation*}
\noindent for any $k>0$. In other words, if we delete $|\Lambda|\left(\frac{8}{\pi}e^{-K^2} + k\frac{2\sqrt{2}}{\sqrt{\pi}}e^{-\frac{K^2}{2}}\right)$ largest phaseless measurements, with probability at least $1-\frac{1}{k^2}$ for the remaining  measurements we would have $|\langle x, \pi (\lambda)g\rangle|\le\frac{K}{\sqrt{M}}$.
\end{proof}



\bibliographystyle{elsarticle-num} 
\section*{\refname}
\bibliography{phase_retrieval_paper_revised}

\end{document}